\documentclass[a4paper,10pt,reqno]{amsart}
\numberwithin{equation}{section}
\usepackage{latexsym,amssymb}
\usepackage{hyperref}
\usepackage{amsmath,amsthm}
\usepackage{amsfonts}
\usepackage[american]{babel}
\usepackage{mathrsfs}
\usepackage{psfrag}
\usepackage{epsfig,inputenc,todonotes}
\usepackage{graphpap,latexsym,epsf,eucal}
\usepackage{color,xcolor}
\usepackage[normalem]{ulem}
\usepackage[eulergreek]{sansmath}

\usepackage{tcolorbox}

\newcommand*{\colorboxed}{}
\def\colorboxed#1#{%
  \colorboxedAux{#1}%
}
\newcommand*{\colorboxedAux}[3]{%
  \begingroup
    \colorlet{cb@saved}{.}%
    \color#1{#2}%
    \boxed{%
      \color{cb@saved}%
      #3%
    }%
  \endgroup
}

\usepackage{amssymb,eucal,paralist,color,enumerate,upgreek}
\usepackage{graphicx}
\usepackage{tikz}

\usepackage{bm}

\usetikzlibrary{
    calc,
    arrows.meta,
    positioning,
    decorations.pathreplacing,
    patterns
}

\newcommand{\R}{\mathbb{R}}
\newcommand{\N}{\mathbb{N}}

\mathchardef\emptyset="001F

\numberwithin{equation}{section}

\newtheorem{theorem}{Theorem}[section]

\newtheorem{lemma}[theorem]{Lemma}
\newtheorem{remark}[theorem]{Remark}

\newtheorem{definition}[theorem]{Definition}
\newtheorem{proposition}[theorem]{Proposition}

\newcommand{\eps}{\varepsilon}

\newcommand{\weakto}{\rightharpoonup} 

\newcommand{\aein}{\text{a.e.\ in }}

\newcommand{\down}{\downarrow}

\newcommand{\dualoperator}

 \def\calE{{\mathcal E}} \def\calF{{\mathcal F}}
  
 \def\calK{{\mathcal K}} \def\calL{{\mathcal L}}
  
 \def\calQ{{\mathcal Q}} \def\calR{{\mathcal R}}

  \def\rmC{{\mathrm C}}
\def\rmD{{\mathrm D}}  
 \def\rmH{{\mathrm H}} 
  \def\rmL{{\mathrm L}}

 \def\rmW{{\mathrm W}}

\def\BS{\boldsymbol} 
\def\BS{\boldsymbol} 

\def\dd{\;\!\mathrm{d}} 

\newcommand{\nchi}{{\raise.2ex\hbox{$\chi$}}}

\definecolor{ddcyan}{rgb}{0,0.1,0.9}
\definecolor{ddmagenta}{rgb}{0.8,0,0.8}
\definecolor{orange}{rgb}{0.6,0.2,0}
\definecolor{vgreen}{rgb}{0.1,0.5,0.2}

\newcommand{\piecewiseConstant}[2]{\overline{#1}_{\kern-1pt#2}}

\newcommand{\upiecewiseConstant}[2]{\underline{#1}_{\kern-1pt#2}}

\newcommand{\piecewiseLinear}[2]{{#1}_{\kern-1pt#2}}

\newcommand{\piecewiseVariational}[2]{\tilde{#1}_{\kern-1pt#2}}

\newcommand{\foraa}{\text{for a.a. }}

\newcommand\JUMP[1]{\mathchoice
                  {\big[\hspace*{-.3em}\big[#1\big]\hspace*{-.3em}\big]}
                   {[\hspace*{-.15em}[#1]\hspace*{-.15em}]}
                   {[\![#1]\!]}
                   {[\![#1]\!]}}

\newcommand{\BV}{\mathrm{BV}}

\newcommand{\Hs}{\mathrm{W}^{1,q}}

\newcommand{\As}{A_q}

 \def\trait #1 #2 #3 {\vrule width #1pt height #2pt depth #3pt}

 \def\fin{\hfill
         \trait .3 5 0
         \trait 5 .3 0
         \kern-5pt
         \trait 5 5 -4.7
         \trait 0.3 5 0
 \medskip}

\newcommand{\RR}{\mathrm{R}}
\newcommand{\resmap}{\mathrm{T}_\eps}
\newcommand{\resop}{\mathbb{T}_\eps}
\newcommand{\res}[1]{\mathfrak{#1}}

\newcommand{\EEE}{\color{black}}

\newcommand{\RED}{\color{red}}

\newcommand{\uu}{\boldsymbol{u}}
\newcommand{\vv}{\boldsymbol{v}}

\newcommand{\effe}{\boldsymbol{f}}

\newcommand{\GC}{\Gamma_{\!\scriptscriptstyle{\rm C}}}

\newcommand{\GDir}{\Gamma_{\!\scriptscriptstyle{\rm Dir}}}
\newcommand{\CLOD}{\overline{\Omega}_{\!\scriptscriptstyle{\rm D}}}
\newcommand{\OD}{\Omega_{\!\scriptscriptstyle{\rm D}}}
\newcommand{\ODeta}{\Omega_{\!\scriptscriptstyle{\rm D}, \eta}}
\newcommand{\ODeps}{\Omega_{\!\scriptscriptstyle{\rm D}, \eps}}

\newcommand{\Omegaep}{\Omega_{+}^\eps}
 \newcommand{\Omegaem}{\Omega_{-}^\eps}
 
\newcommand{\WD}{\rmW_{\!\scriptscriptstyle{\rm D}}}

\newcommand{\ti}{{\times}}

\newcommand{\gr}{\upgamma}

\newcommand{\Xeps}{\boldsymbol{X}_\eps}
\newcommand{\resXeps}{\res{X}_\eps}

\newcommand{\resen}{\res{E}_{\eps_n}}

\newcommand{\resX}{\res{X}}

\newcommand{\TU}{\stackrel{\tau_{U}}{\longrightarrow}}

\newcommand{\weaktoX}{\stackrel{\res{X}}{\rightharpoonup}}

\newcommand{\plc}{\mathfrak{g}}
\newcommand{\plcc}{\mathfrak{h}}
\newcommand{\recop}[2]{\Upsilon^{#1}_{#2}}
\newcommand{\intp}{\boldsymbol{\imath}}

\title[]{From  damage to  delamination via evolutionary Gamma-convergence in a rate-independent  quasibrittle \EEE  regime}

\author{Giovanna Bonfanti, Elisa Davoli, Riccarda Rossi, Marita Thomas}

\begin{document}

\maketitle

\hskip 6.2cm {\it {Dedicated to Fr\'ed\'eric Lebon on the occasion of his 65th birthday}}

\begin{abstract}
 We analyze via Evolutionary Gamma-convergence a stratified composite structure consisting of 
a thin adhesive layer with vanishing thickness and undergoing  rate-independent \EEE damage, as well as two adjacent elastic adherents. As the width of the intermediate layer tends to zero, we prevent complete degradation of the material by assuming that the damage variable scales minimally like the thickness of the adhesive layer. As a result, we identify a limiting model that combines both a brittle constraint and an adhesive-type energy contribution featuring the jump of the admissible displacements. \EEE
\end{abstract}


 \bigskip

\centerline{\today}

\section{Introduction}

In this paper  
we will   deduce a delamination model as the scaling-limit of a volume damage model when the thickness of the damageable layer tends to zero. This damageable layer can be interpreted as the fracture process zone occurring in  \emph{quasibrittle} materials. 
\par
Indeed,  as we will see, 
the resulting limit delamination model features two surface energy terms along 
the damageable  limit interface.  
Firstly, it contains a term that implements the brittle constraint arising in models for brittle, Griffith-type delamination, ensuring that  a diplacement jump is admissible only where the interface has fully damaged,  i.e., where a crack has appeared. 
Additionally, it also contains a second surface energy that accounts for the displacement jumps 
 across the interface,  as in cohesive zone models.
But, unlike  the standard  cohesive zone models in mathematical literature, the dissipation potential of our limit model does not  track the maximal crack opening. 
In order to distinguish the model discussed in this work from the models for cohesive and brittle delamination, as well as  from the models for \EEE adhesive contact,   while stressing    its characteristic features, we  have opted for  calling it   a \emph{quasibrittle} delamination model. \EEE

\subsection{An overview of bulk-to-surface limits for evolution problems in damage and delamination}
Over the past decades, laminated composites have become widely used in engineering, with applications ranging from aerospace and automotive to civil structures. These systems are typically assembled by bonding, where thin adhesive layers play a key role in ensuring structural continuity, distributing loads, and reducing stress concentration. 
However, the mechanical behavior of these interphases is a critical issue: damage can arise and propagate at multiple spatial scales. 
Microscopic phenomena such as matrix cracking, fiber breakage, and fiber-matrix debonding can manifest at the macroscopic scale as crack propagation, delamination, and loss of structural stability.
Accurately modeling damage in these interphases is therefore essential to describe the mechanical behavior of the whole structure. At the same time, the small thickness of adhesive layers suggests the usage of dimension reduction techniques to investigate their asymptotic behavior as their thickness tends to zero.
From a mathematical point of view, the derivation of lower dimensional models for thin structures starting from their three dimensional counterparts has been tackled in the literature using different analytical
tools and modeling approaches. 

In \cite{BBLR17,BBL19,RLR22}, formal asymptotic expansions have been used to derive imperfect interface models for  adhesive contact between (thermo)elastic solids. In the limiting configuration, the thin layer reduces to a contact surface, where the mechanical interaction between the adherents is characterized by so-called \emph{imperfect} transmission conditions and the evolution of a surface damage parameter is governed by a suitable flow rule. 
The same asymptotic approach has been applied to a more general setting including hyperelastic \cite{RLR18} and piezoelectric \cite{RSRL23} materials undergoing micro-cracking and damage evolution, as well as anisotropic materials characterized by asymmetric behavior in tension and compression \cite{SRL25,LR22}.
This formal method proves particularly valuable when strong nonlinearities prevent a fully rigorous analytical treatment, still providing meaningful insight into the asymptotic behavior of the overall system.

A different approach to the study of bulk-to-surface limits in damage 
is based on variational techniques and,  in particular, on \emph{$\Gamma$-convergence},  suitably adapted to evolutionary problems. These tools were pioneered, in the context of rate-independent (or \emph{quasistatic}) evolution in \cite{Giac05ATAQ}, where the limit passage from bulk to surface
damage (in particular, brittle fracture) was addressed via the Ambrosio-Tortorelli approximation of the Mumford-Shah functional. Let us mention that this kind of  asymptotic analysis has been recently extended to (rate-independent) \emph{cohesive} fracture in \cite{BoCoIUR}. 
\par
The  seminal paper \cite{MRS06} settled the machinery for systematically applying $\Gamma$-convergence techniques in the context of asymptotic analyses for rate-independent systems (and beyond, see \cite{Mielke-evol}). Ever since,  the tools  for \emph{Evolutionary $\Gamma$-convergence} have been widely applied; in the context of delamination problems, we recall \cite{RoScZa09QDP}, dealing with the limit passage from adhesive to brittle delamination, as well as the dimension reduction analyses  in \cite{FreParRouZan11,FreRouZan13},  where models for both adhesive contact and brittle delamination in 2D plates have been  derived as limits of delamination between 3D thin plates (in the recent \cite{BDR26}, some of those results have been extended to the case where the rate-independent delamination flow rule is coupled with a rate-dependent momentum  balance). 

While the dimension reduction analysis in  \cite{FreParRouZan11,FreRouZan13,BDR26} only concerns delamination processes, both on the level of the fixed and on that of the vanishing thickness problem, the problem tackled in this paper has a different flavour. 
Our investigation
indeed proceeds from \cite{MiRoTh10DDNE}, where a dimension reduction analysis  leading  from bulk damage to surface delamination has been rigorously carried out, again  in the case where both processes are treated as \emph{rate independent}. 
Because of that feature, low temporal regularity is expected for the internal variable describing the inelastic phenomenon (i.e., the damage and the delamination parameters, respectively) and, thus, the related PDE systems need to be weakly formulated. In \cite{MiRoTh10DDNE},  the concept of 
\emph{Energetic solution}
\cite{MielkeRoubicek15} was adopted, and, correspondingly, the technique of \emph{Evolutionary $\Gamma$-convergence}  was leveraged for the asymptotic analysis. The main results in \cite{MiRoTh10DDNE} state that 
 Energetic solutions to the system for isotropic damage converge to an Energetic solution for a brittle delamination model
 as the thickness of the thin layer between the two bulk bodies, subject to damage process, tends to zero.

\par
  In this paper we aim to extend the analysis carried out  \cite{MiRoTh10DDNE}
to a specific regime therein not considered.

\subsection{Our result}
We place our investigation in the very same setup considered in \cite{MiRoTh10DDNE}.  Thus, we  consider
 a stratified structure consisting of two elastic adherents bonded by a thin adhesive layer; we suppose that only in the thin layer (rate-independent) damage occurs, as shown in Figure \ref{fig1}.

\begin{figure}
\begin{tikzpicture}[
    scale=0.95,
    line join=round,
    line cap=round,
    >=Latex,
    every node/.style={font=\small}
]


\def\W{6.0}      
\def\H{2.0}      
\def\dx{0.8}     
\def\dy{1.4}     



\newcommand{\DrawCuboid}[8]{%

\begin{scope}[shift={({#1},{#2})}]


\coordinate (A) at (0,0);
\coordinate (B) at (\W,0);
\coordinate (C) at (\W,\H);
\coordinate (D) at (0,\H);

\coordinate (A2) at (\dx,\dy);
\coordinate (B2) at ($(B)+(\dx,\dy)$);
\coordinate (C2) at ($(C)+(\dx,\dy)$);
\coordinate (D2) at ($(D)+(\dx,\dy)$);


\draw[thick]
(A)--(B)--(C)--(D)--cycle;

\draw[thick]
(D)--(D2)--(C2)--(C);

\draw[thick]
(B)--(B2)--(C2);

\draw[dashed]
(A)--(A2)--(D2);

\draw[dashed]
(A2)--($(A2)+(4.8,0)$);


\pgfmathsetmacro{\xL}{#3-#4/2}
\pgfmathsetmacro{\xR}{#3+#4/2}

\coordinate (E1) at (\xL,0);
\coordinate (E2) at (\xR,0);

\coordinate (F1) at (\xL,\H);
\coordinate (F2) at (\xR,\H);

\coordinate (E1b) at ($(E1)+(\dx,\dy)$);
\coordinate (E2b) at ($(E2)+(\dx,\dy)$);

\coordinate (F1b) at ($(F1)+(\dx,\dy)$);
\coordinate (F2b) at ($(F2)+(\dx,\dy)$);

\fill[gray!45]
(E1)--(E2)--(F2)--(F1)--cycle;

\fill[gray!25]
(F1)--(F2)--(F2b)--(F1b)--cycle;

\fill[gray!55]
(E2)--(F2)--(F2b)--(E2b)--cycle;

\draw
(E1)--(F1)
(E2)--(F2)
(F1)--(F1b)
(F2)--(F2b)
(E2)--(E2b);

\draw[dashed]
(E1)--(E1b)
(E1b)--(F1b);


\node at (1.9,2.55) {$#5$};
\node at (4.8,2.55) {$#6$};

\node at (#3,2.45) {$#7$};

\node[fill=white,inner sep=1pt]
at (0.45,1.72)
{$\Gamma_{\mathrm{Dir}}$};

\node[fill=white,inner sep=1pt]
at (6.25,1.72)
{$\Gamma_{\mathrm{Dir}}$};

\node[font=\Large]
at (-0.7,3.15)
{#8};

\end{scope}
}


\DrawCuboid
{0}{4.2}
{3.0}{0.5}
{\Omega_{-,\eps}}
{\Omega_{+,\eps}}
{\Omega_{D,\eps}}






\draw[<->] (0,3.85)--(2.75,3.85);
\node at (1.4,3.55) {$L-\eps$};

\draw[<->] (2.75,3.85)--(3.25,3.85);
\node at (3.0,3.55) {$2\eps$};

\draw[<->] (3.25,3.85)--(6.0,3.85);
\node at (4.7,3.55) {$L-\eps$};

\foreach \x in {0,2.75,3.25,6}
{
    \draw (\x,4.05)--(\x,3.7);
}



\begin{scope}[shift={(12.6,5)}]

\draw[->,thick] (0,0)--(1.0,0)
node[right] {$x_1$};

\draw[->,thick] (0,0)--(0,1.3)
node[above] {$x_2$};

\draw[->,thick] (0,0)--(-0.45,-0.8)
node[left] {$x_3$};

\end{scope}

\end{tikzpicture}
\caption{The set $\Omega$. The two adherents $\Omega_{-,\eps}$ and $\Omega_{+,\eps}$ are joined by an adhesive layer $\ODeps$ undergoing rate-independent damage. 
 The tikz code for this picture was created with the help   of AI. }
 \label{fig1}
\end{figure}
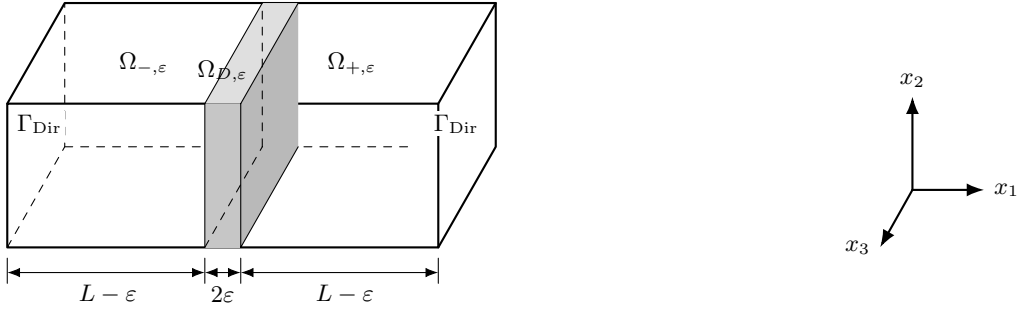

\par
In the reference configuration $\Omega = \Omega_{-,\eps}{\cup} \ODeps {\cup} \Omega_{+,\eps} $, where the layer $ \ODeps $ has \emph{fixed} thickness $\eps>0$,
we consider the rate-independent evolution of a damage process but nonetheless 
 exclude complete degradation of the material.  The mechanical behaviour of the system is governed by the elastic equilibrium equation for the displacement $\uu : \Omega{\times}(0,T) \to \R^d$, with the stress tensor depending on the \emph{full} displacement gradient $\nabla \uu$. Following the approach by \textsc{M.\ Fr\'emond}  \cite{FreNed96DGDP,Fre02},  damage (which only occurs in $ \ODeps $)  is modeled by means of an internal variable 
 $z:\ODeps\times (0,T)\to [0,1]$, such that $z=0$ corresponds to complete damage, $z=1$ to a fully intact material, and 
intermediate values $0<z<1$ describe partial damage. We postpone to Section \ref{s:2} the precise formulation of the PDE system for bulk damage 
and its weak formulation in terms of the Energetic concept: here, we will only comment on the energetics underlying the model that is indeed at the core of its Energetic formulation. In fact, the evolution of damage results from the interplay between the mechanisms of energy conservation and energy  dissipation. The latter (note that only the internal variable dissipates energy) is  encoded by the 
$1$-homogeneous dissipation potential
\[
\calR = \calR(\partial_t z) =\begin{cases}
\int_{\ODeps} \kappa |\partial_t z| \dd z  & \text{if } \partial_t z \leq 0 \ \aein \, \ODeps,
\\
\infty  & \text{else}
\end{cases} \qquad\text{with }\kappa>0
\]
which, in particular, enforces unidirectional evolution of damage, as $t\mapsto z(t,\cdot)$ is non-increasing. 
The time-dependent energy functional driving the process
consists of two contributions
\[
\calE = \calE(t,u,z) = \mathcal{Q}_{\mathrm{el}}(t,\uu,z) +\mathcal{F}(z)\,.
\]
We postpone to Section \ref{s:2}
 the precise definition of the surface energy $\mathcal{F}$ and instead focus on 
 the elastic energy functional 
 \[
 \mathcal{Q}_{\mathrm{el}} (t,\uu, z) = \mathcal{W}^-(t,\uu)+\mathcal{W}_{\!\scriptscriptstyle{\rm D}}(\uu, z) +\mathcal{W}^+(t,\uu), 
 \]
 where the terms  $\mathcal{W}^{\pm}$ encode the elastic energy (and the time-dependent force) in the domains $\Omega_{\pm,\eta}$, while the elastic energy in  the damageable layer $\ODeps$ is given by 
 \[
 \mathcal{W}_{\!\scriptscriptstyle{\rm D}}(\uu, z)  = \int_{\ODeps}  \tfrac{\mu}2 \EEE \gr(z) |\nabla \uu|^2 \dd x \,,
 \]
  with $\mu>0$ a fixed positive constant and $\gr $ a continuous function. \EEE
Complete damage is ruled out in that we assume that the material preserved its elastic properties regardless of $z$ approaching zero, 
i.e., we ask for the lower bound 
\begin{equation}
\label{lower-bound-intro}
\gr(z) \geq \eps\,.
\end{equation}
Notably, the lower bound on the damage variable is chosen to degenerate in the limit, thereby  giving rise to a limit model in which complete damage is achieved.
\par
Complete damage was also excluded in \cite{MiRoTh10DDNE} by 
 imposing a lower bound on the damage variable, specifically $z \in [\eps^\gamma, 1]$ for a parameter $\gamma \in (0,\infty){\setminus}\{1\}$.   Instead, \eqref{lower-bound-intro}
 corresponds to constraining  the values of $z $ to  the interval $[\eps,1]$.  \EEE 
This assumption, combined with a specific choice for the elastic energy density, places our work in a setting not explored in \cite{MiRoTh10DDNE}, as discussed in further detail in Remark \ref{on-greps} ahead. 
\par
 In \cite{MiRoTh10DDNE} the authors fully identified the regimes in which $z$ was taking values in $[\eps^{\gamma},1]$, showing, for $\gamma<1$ a complete delamination excluding displacement jumps, and obtaining, for $\gamma>1$ a brittle constraint but no explicit jump terms in the limiting energetics of the model.  The main result of this paper  is the analysis of the  critical \EEE regime $\gamma=1$.  Namely, in \underline{\textbf{Theorem  \ref{thm:main}}} ahead we will show that,
 Energetic solutions to the rate-independent system for bulk damage converge, as $\eps \down 0$, to Energetic solution of a delamination system driven by an energy functional that,
  in this critical regime $\gamma=1$,
 \begin{itemize}
 \item[-]
 on the one hand features the brittle constraint
 \[
 \res{z} \JUMP{\uu} =0 \text{ on }\GC
 \]
 (with $\res{z}$ a suitable rescaling of the delamination variable for the limit system),  and 
 thus allows for fractures in the interface $\GC$, 
 \item[-] 
 but on the other hand still also penalizes displacement jumps across the crack via the \emph{adhesive-type} energy contribution 
 \[
 \int_{\GC}   \tfrac{\mu}2    |\JUMP{\uu}|^2\,\mathrm{d}x\,.
 \]
 \end{itemize}
 \par 
 Because of this concurrence of brittle- \emph{and} adhesive-type features, we have opted for calling this limit model \emph{quasibrittle}. Indeed, 
according to \cite{BaLeSa22QFMS},
 quasibrittle fracture, also called `cohesive softening' fracture,
  occurs in intrinsically heterogeneous materials like concrete or mortars. They have a long and wide nonlinear 
   fracture process zone in front of the crack tip, that is not plastic but undergoes progressive softening damage in the form of randomly distributed microcracking, frictional micro-slips, and grain interlock. Moreover, \cite{BaLeSa22QFMS} highlights that, when the structure size is far larger than the maximum inhomogeneity size, every quasibrittle structure becomes perfectly brittle. Conversely, when
    a perfectly brittle homogeneous structure becomes sufficiently small,  as occurs in micro- or nano-scale devices, it transits to a quasibrittle fracture. Due to the occurring process zone, quasibrittle fractures are often modeled in engineering with the aid of cohesive zone models, containing an additional surface energy term that accounts for a traction-separation law featuring the crack-opening, i.e., the displacement jump $\JUMP{\uu}$ across the crack lips, cf.\ e.g.\  \cite{OrtPan99FDIC}.
However, in mathematical literature, cohesive zone fracture models often contain a specific dissipation law related to locally tracking the maximal crack opening, cf.\ e.g.\  \cite{DMZan07QSCG,Cagnetti08,Cagnetti-Toader11, LarsenSla14, NegriScala17, CLO18, ThoZan17CZDV}.
To our knowledge, the co-existence of adhesive and brittle features
 that characterizes the limit model obtained in this paper has not yet been observed elsewhere. 
\par
 Finally, let us emphasize that  depending on the size of the parameter $\mu>0$,  displacement jumps across the crack will be possible or actually rather still be prevented in the model. Hence,  we reckon that, in the case $\gamma=1$, our limit model   indeed effectively describes the transition between the regime $\gamma<1,$ where  displacement jumps are not featured,  and the regime $\gamma>1,$  where displacement jumps are allowed. 
\par
 From the mathematical standpoint, let us mention that   \EEE the proof of Theorem  \ref{thm:main} follows the general strategy designed in  \cite{MRS06} 
for abstract rate-independent systems, in particular adapting the arguments from \cite{MiRoTh10DDNE} to the scaling $\gamma=1$, which leads to the emergence of the adhesive contribution. 

\noindent
The \underline{\bf plan of the paper} is as follows. In Section  \ref{s:2} we collect our assumptions on the reference configuration, introduce a weak solvability concept, i.e. Energetic solutions, and establish a related existence result, Theorem \ref{thm:Marita}. Section \ref{s:3} sets the stage for the dimensional reduction analysis and states the main result of the paper, Theorem  \ref{thm:main}. Sections \ref{s:4} and  \ref{s:5} are devoted to the proof of Theorem  \ref{thm:main},   which proceeds through several steps verifying the conditions that guarantee the Evolutionary $\Gamma$-convergence of the damage rate-independent system to the system describing delamination. \EEE


\EEE

%
\section{The damage system in the volume domain \EEE}
\label{s:2}
The forthcoming analysis is set in any dimension $d\geq 2$, although the physically reasonable cases are $d\in \{2,3\}$. \EEE 

Following 
\cite{MiRoTh10DDNE}, we consider 
a domain 
\[
\Omega = (-L,L)\times 
\GC, \qquad \text{with  } \GC=(-H,H)^{d-1}
\]
for given $H,\, L >0$, 
   which is the union of  three cuboid-type Lipschitz-domains
\[
\Omega = \Omega_{-,\eta} \cup \ODeta \cup \Omega_{+,\eta}
\]
with $\Omega_{-,\eta} = (-L,-\eta) \times \GC$,  $\ODeta = (-\eta,\eta)\times \GC$, $\Omega_{+,\eta} = (\eta,L) \times \GC$, 
and $\eta \in (0,L)$  a 
\emph{fixed} parameter.
We will denote by 
$\Gamma_{\mathrm{Dir}}$  the following portion of the outer boundary
\[
\Gamma_{\mathrm{Dir}} = \{{-}L\} {\times} \GC \cup \{L\} {\times} \GC \,.
\]
The domains $\Omega_{\pm, \eta}$ are occupied by a non-linearly elastic material that is damage-resistive, while $\ODeta$ is occupied by a 
  material undergoing  \emph{rate-independent} damage (see Figure \ref{fig1}). \\

The mechanical behavior of the system is governed by an elastic equilibrium equation for the displacement field 
$\uu:\Omega\times (0,T)\to\RR^d$, with the  stress tensor depending on the displacement gradient $\nabla \uu$. 
Damage is confined to the intermediate layer where,
following the approach in \cite{FreNed96DGDP,Fre02}, the momentum balance equation is coupled with a flow rule for the damage variable $z:\ODeta\times (0,T)\to [0,1]$. We recall that this scalar field characterizes the damage state of the material: $z=0$ corresponds to complete damage, $z=1$ to a fully intact material, and 
intermediate values $0<z<1$ describe partial damage.  
\par
More precisely, \EEE
for fixed  $\eta \in (0,L)$, we consider the following damage system: 
\begin{subequations}
\begin{itemize}
\item[$\boldsymbol{-}$] 
In $ \Omega_{\pm,\eta}$  
the stored energy density $\rmW: \R^{d{\times}d} \to \R$, cf.\ \eqref{Wgeneral} ahead, only depends  on $\nabla \uu$, which reflects the 
 fact that damage occurs only on $\ODeta$:
\label{unrescaled}
\begin{align}
\label{bulk-momentum}
\begin{cases}
-\mathrm{div}  (\rmD \rmW (\nabla \uu) )  =\effe \qquad \text{ in } \Omega_{-,\eta}  \ti (0,T) \,,
\\
-\mathrm{div}  (\rmD \rmW(\nabla \uu) )=\effe \qquad \text{ in } \Omega_{+,\eta}  \ti (0,T) \,,
\end{cases}
\end{align}
where $\effe$ is a time-dependent applied volume force. \EEE
\item[$\boldsymbol{-}$] 
 The elastic equilibrium equation on $\ODeta$ instead takes the form 
\begin{align}
\label{damage-momentum}
-\mathrm{div} ( \mu \EEE \gr(z)   \nabla \uu ) =  {\bf 0} \qquad \text{ in } \ODeta {\times} (0,T),
\end{align}
corresponding to a stored elastic energy density
depending on $z$, \EEE
through the continuous function $\gr$, 
 and on  $F = \nabla \uu$, with  quadratic growth w.r.t.\ $F$. 
\item[$\boldsymbol{-}$]  On the Dirichlet part  
   of the outer boundary we take  (for simplicity) zero Dirichlet boundary conditions
\begin{equation}
\label{DirBC}
\uu =0 \qquad \text{on } \Gamma_{\mathrm{Dir}} {\times} (0,T)\,.
\end{equation}
\item[$\boldsymbol{-}$] 
The elastic equilibrium equation \eqref{damage-momentum} is 
 coupled with the (rate-independent) flow rule for the damage parameter:
\begin{align}
\label{flow-rule}
\partial \mathrm{R}(z_t) + \As z +Az + \partial\upphi(z) \ni -   \frac\mu 2  \EEE  \gr'(z)     \nabla \uu {\colon} \nabla \uu   \EEE      
 \qquad \text{ in } \ODeta \times (0,T),
\end{align}
supplemented with no-flux b.c.\ for $z$ on $\partial  \ODeta \times (0,T)$. 
Here,
and henceforth, we write $z_t$ in place of $\partial_t z$. \EEE
The density of the dissipation potential  $\mathrm{R} : \R \to [0,\infty]$ is given by
\begin{equation}
\label{dissRdens}
\mathrm{R}(v): = \begin{cases} \kappa |v| & \text{if } v \leq 0,
\\
\infty & \text{otherwise},
\end{cases}
\end{equation}
with $\kappa$ a fixed positive constant.  Moreover, the function $\upphi$,  with $\mathrm{dom}(\upphi)\subset [0,1]$,   renders the physical constraint on the variable $z$ and takes into account possible convex contributions in the potential energy (cf.\ \eqref{upphi} and \eqref{potential-F} ahead). \EEE    While  
 $A: \mathrm{H}^1(\ODeta)\to \mathrm{H}^1(\ODeta)^* $  is the Laplacian operator with zero Neumann boundary conditions, 
$\As : \Hs(\ODeta)\to \Hs(\ODeta)^* $ is the corresponding $q$-Laplacian operator, with $q>d$.
\end{itemize}
\end{subequations}
  We emphasize that the flow rule \eqref{flow-rule} is only formally written: for system \eqref{unrescaled}
we shall indeed address the well-known weak solvability concept of \emph{Energetic  solutions}, and carry out our asymptotic analysis within that framework. 

\begin{remark}
\upshape
We note that, in principle,  the  damage model could involve an elastic contribution of the form  
$$\int_{\ODeta}  \tfrac12 \gr(z)  \mathbb{C}  \nabla \uu {:} \nabla \uu \dd x\,,$$ where   $\mathbb{C} \in \R^{d{\times} d}$   is a general elasticity tensor. 
For the sake of the analysis, however, we shall restrict to the case 
in which  $\mathbb{C} $ coincides with  $\mu \mathbb{I}$, $ \mathbb{I}$ being \EEE the identity tensor.
\par
 In turn, our analysis could be extended in other directions. 
Let us also mention that some generalizations of our framework are possible. First, we could treat time-dependent Dirichlet loads, which we will omit here to avoid overburdening the exposition. Second, following \cite{MiRoTh10DDNE}, 
we could consider an asymmetric material response to the damage in compression and tension, which would ultimately enforce the non-interpenetration condition for the limiting model.  We prefer, however, to neglect this refinement in order to focus on the main mathematical novelty of the paper.
\end{remark} \EEE



\subsection{Energetics}
\ First of all, let us specify our standing assumptions on the constitutive functions of the damage system and on the time-dependent loadings. 
Throughout the paper, we shall suppose that 
\begin{enumerate}
\item   The elastic energy density 
$\rmW: \R^{d{\times}d} \to \R$ is convex and with quadratic growth
\begin{equation}
\label{Wgeneral} 
\exists\, c_{\rmW}, \tilde{c}_{\rmW}, C_{\rmW}>0  \ \forall\, F \in \R^{d{\times}d} \, : \qquad 
c_{\rmW} |F|^2 - C_{\rmW} \leq W(F) \leq \tilde{c}_{\rmW} |F|^2 + C_{\rmW}\,,
\end{equation}
while the energy density in the damageable domain,  $\WD: \R{\times} \R^{d{\times}d} \to \R$, is  given by 
\begin{equation}
 \WD(z,F)=   \tfrac\mu 2 \EEE  \gr(z)  F{\colon}F \qquad \text{with } \gr \in \rmC^0(\R) \text{ such that } \exists\, c_{\gr}>0 \ \forall\, z \in \R \, : \ \gr(z) \geq  c_{\gr}\,.
\end{equation}

\EEE

\item
The function $\upphi: \R \to \R$ is
\begin{equation} 
\label{upphi}
\text{lower semicontinuous, convex, with proper domain } \mathrm{dom}(\upphi) \subset [0,1].
\end{equation}
\item The volume  force $\effe$ satisfies
 \begin{equation}
\label{force0}
\effe \in \mathrm{C}^1([0,T];\mathrm{H}^1(\Omega;\R^d)^*) \qquad \text{with } \ \mathrm{supp} (\effe(t)) \cap   \overline{\Omega}_{\!\scriptscriptstyle{\rm D}, \eta} = \emptyset   \text{ in the sense of distributions for all } t \in [0,T].
 \end{equation}
 We highlight that, in \cite{MiRoTh10DDNE}   the analogue of the second condition was required for the time-dependent Dirichlet loading. \EEE
\end{enumerate}
  For later use, we remark that, as consequence of the support condition in \eqref{force} we have 
 \begin{equation}
 \label{partial-control-eta}
  \langle \effe(t), \uu \rangle_{\mathrm{H}^1(\Omega;\R^d)} \leq \|\effe\|_{\mathrm{L}^\infty(0,T;\mathrm{H}^1(\Omega;\R^d)^*) } \| \uu \|_{\mathrm{H}^1(\Omega_{-,\eta}{\cup}\Omega_{+,\eta};\R^d)}\,.
 \end{equation}


 The elastic energy  in the non-damageable domains $\Omega_{\pm,\eta}$ is encoded by \EEE   $\mathcal{W}^{\pm }: [0,T]\ti \mathrm{H}^1(\Omega_{\pm,\eta} ;\R^d) \to  \R \EEE $
\[
\mathcal{W}^+ (t,\uu) =  \int_{\Omega_{+,\eta}  }  \rmW(\nabla \uu\EEE) \dd x  - \langle \effe(t), \uu \rangle_{\mathrm{H}^1(\Omega_{+,\eta};\R^d) )} , \quad \mathcal{W}^- (t,\uu) =  \int_{\Omega_{-,\eta}  }  \rmW( \nabla \uu \EEE) \dd x  - \langle \effe(t), \uu \rangle_{\mathrm{H}^1(\Omega_{-,\eta};\R^d) )} 
\]
as well as $\mathcal{W}_{\!\scriptscriptstyle{\rm D}} : \mathrm{H}^1(\ODeta) \ti   \mathrm{L}^\infty (\ODeta) \to [0,+\infty)$
\begin{equation}
\label{WD}
\mathcal{W}_{\!\scriptscriptstyle{\rm D}}(\uu, z) : =  \int_{\ODeta}  \WD(z,\nabla \uu )  \dd x\,.
\end{equation}
Hence, the elastic 
energy 
functional $\mathcal{Q}_{\mathrm{el}} : [0,T]\ti \mathrm{H}^1(\Omega;\R^d) \ti \mathrm{L}^\infty (
\ODeta) \to \R$ is given by 
\begin{equation}
\label{energy-eps}
\mathcal{Q}_{\mathrm{el}}(t,\uu,z) : =\mathcal{W}^+ (t,\uu) + \mathcal{W}^- (t,\uu) + \mathcal{W}_{\!\scriptscriptstyle{\rm D}}(\uu, z)  \,.
\end{equation}
 The damage flow rule \eqref{flow-rule} is a balance law between the frictional forces in $\partial \mathrm{R}(z_t) $ and the potential restoring forces associated with $z \mapsto \mathcal{W}_{\!\scriptscriptstyle{\rm D}}(z,\uu)$ and with the potential energy $\mathcal{F}:  \Hs(\ODeta)  \to \R{\cup}\{\infty\} $
\begin{equation}
\label{potential-F}
\mathcal{F}(z):= \begin{cases}   \int_{\ODeta}\left( \frac1q \EEE |\nabla z|^q{+}\frac12 |\nabla z|^2{+} \upphi(z) \right) \dd x  & \text{if }  \upphi(z)  \in \mathrm{L}^1(\ODeta),
\\
\infty & \text{otherwise.}
\end{cases}
\end{equation}
\subsection{Energetic solutions}
In order to introduce in a compact form our weak solvability concept for system \eqref{unrescaled}, with fixed $\eta>0$, we introduce 
\begin{itemize}
\item[-]  the state space 
\[
\boldsymbol{X}:= 
\mathrm{H}^1(\Omega;\R^d) \ti \Hs (
\ODeta);
\]
\item[-]  the \emph{driving energy} functional for the damage system 
 \[
 \calE : [0,T]\ti \boldsymbol{X} \to \R, \qquad 
 \calE(t,u,z) : = \mathcal{Q}_{\mathrm{el}}(t,\uu,z) +\mathcal{F}(z)\,.
 \]
 \item[-] 
 the $1$-homogeneous dissipation potential
\begin{equation}
\label{calR-potential}
 \mathcal{R}: \mathrm{L}^1(\ODeta) \to [0,\infty], \qquad 
\mathcal{R} (v) = \int_{\ODeta} \mathrm{R}(v(x)) \dd x \,,
\end{equation}
originating from  the dissipation density $\mathrm{R}$ in \eqref{dissRdens}.
\end{itemize}
We will refer to the triple $(\boldsymbol{X}, \calE,\calR) $ as a \emph{rate-independent system for damage}. 

Observe that $\calR$ induces the total variation functional
\[
\mathrm{Var}_{\calR}(\zeta;[a,b]): = \sup \left\{ \sum_{i=0}^N \calR(\zeta(\tau_i){-} \zeta(\tau_{i-1})) \, : \ (\tau_i)_{i=0}^N \in \mathscr{P}([a,b]) \right\}, 
\]
for a given function $\zeta: [a,b]\to \mathrm{L}^1(\ODeta)$, with $ \mathscr{P}([a,b])$ the set of partitions of the interval $[a,b]$. 
\par
We are now in a position  to specify our weak solvability concept for \eqref{unrescaled}.
\begin{definition}
Let $\eta>0$ be fixed. We call a pair $(\uu,z)$, with 
\begin{equation}
\label{reg-uz}
\uu\in \mathrm{L}^\infty (0,T; \mathrm{H}^1(\Omega;\R^d)), \qquad z \in  \mathrm{L}^\infty (0,T; \Hs(\ODeta)){\cap} \BV ([0,T]; \mathrm{L}^1(\ODeta))
\end{equation}
an \emph{Energetic solution} to the rate-independent  system for damage  $(\boldsymbol{X}, \calE,\calR) $   if it complies with 
\begin{subequations}
\label{enform-dam-fixed}
\begin{itemize}
\item[--] the  \emph{global stability condition}
\begin{equation}
\label{glSTAB}
\calE(t,\uu(t), z(t)) \leq \calE(t,\hat{\uu}, \hat{z})  + \calR(\hat{z}{-} z(t)) \qquad \text{for all } (\hat{\uu},   \hat{z})  \in  \boldsymbol{X};
\end{equation}
\item[--] the  \emph{global energy balance}
\begin{equation}
\label{glEnBal}
\calE(t,\uu(t), z(t))+\mathrm{Var}_{\calR}(z;[s,t])  = \calE(t,\uu(s), z(s))+ \int_s^t \partial_t \calE(r,\uu(r), z(r)) \dd r  \qquad \text{for all } [s,t]\subset [0,T]\,.
\end{equation}
\end{itemize}
\end{subequations}
\end{definition}
For later use, it is useful to introduce the \emph{stable set} 
\[
\boldsymbol{S}:= \bigcup_{t\in [0,T]}\mathrm{S}(t), \qquad \mathrm{S}(t):= \{ (\uu,z)\in \boldsymbol{X}\, : \ \calE(t,\uu, z) \leq \calE(t,\hat{\uu}, \hat{z})  + \calR(\hat{z}{-} z) \  \forall\, (\hat{\uu},   \hat{z})  \in  \boldsymbol{X}\}\,.
\]
\begin{remark}[Reformulation of the energy balance]
\sl
For a given  Energetic solution we have $\mathrm{Var}_{\calR}(z;[0,T]) $; since $\calR$ controls $\| \cdot\|_{\mathrm{L}^1(\ODeta)}$, 
this yields $z 
\in \BV ([0,T]; \mathrm{L}^1(\ODeta))$.  Taking into account the unidirectionality constraint encompassed in the definition of $\calR$, we then obtain that 
\[
\forall\, 0\leq s \leq t \leq T \,: \ z(t,x) \leq z(s,x) \qquad  \foraa\, x \in \ODeta\,.
\]
Therefore, it is not difficult to check that  $\mathrm{Var}_{\calR}(z;[s,t])  = \kappa \int_{\ODeta} (z(s,x){-}z(t,x)) \dd x $. All in all, \eqref{glEnBal} explicitly rewrites as
\begin{equation}
\label{glob-en-bal-expl}
\mathcal{Q}_{\mathrm{el}}(t,\uu(t),z(t)) + \mathcal{F}(z(t)) + \kappa  \int_{\ODeta} (z(s,x){-}z(t,x)) \dd x= \mathcal{Q}_{\mathrm{el}}(s,\uu(s),z(s)) + \mathcal{F}(z(s)) 
-\int_s^t  \langle \effe'(r), \uu(r) \rangle_{\mathrm{H}^1(\Omega;\R^d)} \dd r 
\end{equation}
for all $[s,t]\subset [0,T]$. 
\end{remark}
\par 
The following result states the existence of Energetic solutions to the rate-independent system for bulk damage, and was proven in \cite[Thm.\ 3.2.7]{MT:Diss}. 
\begin{theorem}
\label{thm:Marita}
Let $\eta>0$ be fixed. For any 
\begin{subequations}
\label{init-conditions}
\begin{equation}
\label{init-data}
(\uu_0, z_0) \in \mathrm{H}^1(\Omega;\R^d) \ti  \Hs(\ODeta)
\end{equation}
fulfilling the stability condition at time $t=0$
\begin{equation}
\label{stab-cond-init}
(\uu_0, z_0) \in \mathrm{S}(0)
\end{equation}
\end{subequations}
there exists an Energetic solution $(\uu,z)$ to the rate-independent system $(\boldsymbol{X}, \calE,\calR) $,   fulfilling  the initial condition $(\uu(0),z(0))=(\uu_0,z_0)$.
\end{theorem}
\begin{remark}
\label{rmk:more-precise}
\sl 
We stress that \cite[Thm.\ 3.2.7]{MT:Diss} applies to a broader range of (isotropic) damage models, as it allows for elastic energy densities 
$\rmW$ and $\WD$
compatible with geometrically nonlinear models (at finite strains). 
While finite-strain elasticity is out of the reach of our techniques for  dimensional-reduction analysis, we could profit by the generality of  \cite[Thm.\ 3.2.7]{MT:Diss}  and 
 extend the present study to the case  in which the elastic energy density $\rmW$ in the \emph{undamageable} domains $\Omega_{\pm,\eta} $ has general growth $p \geq 2$
 (as it will be clear from the analysis in  Sections  \ref{s:5} and \ref{s:6}), \EEE
  instead  of simply being quadratic as  in \eqref{Wgeneral}. In the forthcoming analysis, we will stick to the quadratic case only  to avoid overburdening the exposition. \EEE
\end{remark}

\section{Our main result}
\label{s:3}
In this section, we introduce the framework for our asymptotic analysis  specifying the energetic setup along with the choice of rescaling factors for the various contributions to the driving energy functional. The damage variable will be rescaled via a classical change of variables, and the damage flow rule will be  formulated in an auxiliary domain of fixed thickness. The undamageable domains and the displacement variable will be  left unrescaled, following the approach of \cite{MiRoTh10DDNE}. By resorting to the theory of Evolutionary $\Gamma$-convergence, we will  establish here our main convergence result, Theorem~\ref{thm:main}. This theorem guarantees convergence to  Energetic solution to a rate-independent delamination system that incorporates both the brittle constraint and an adhesive-type energy contribution.

\EEE

\subsection{Setup for the dimension reduction}

 We will  now address the case in which the thickness of the damageable domain vanishes.\\ 
 
 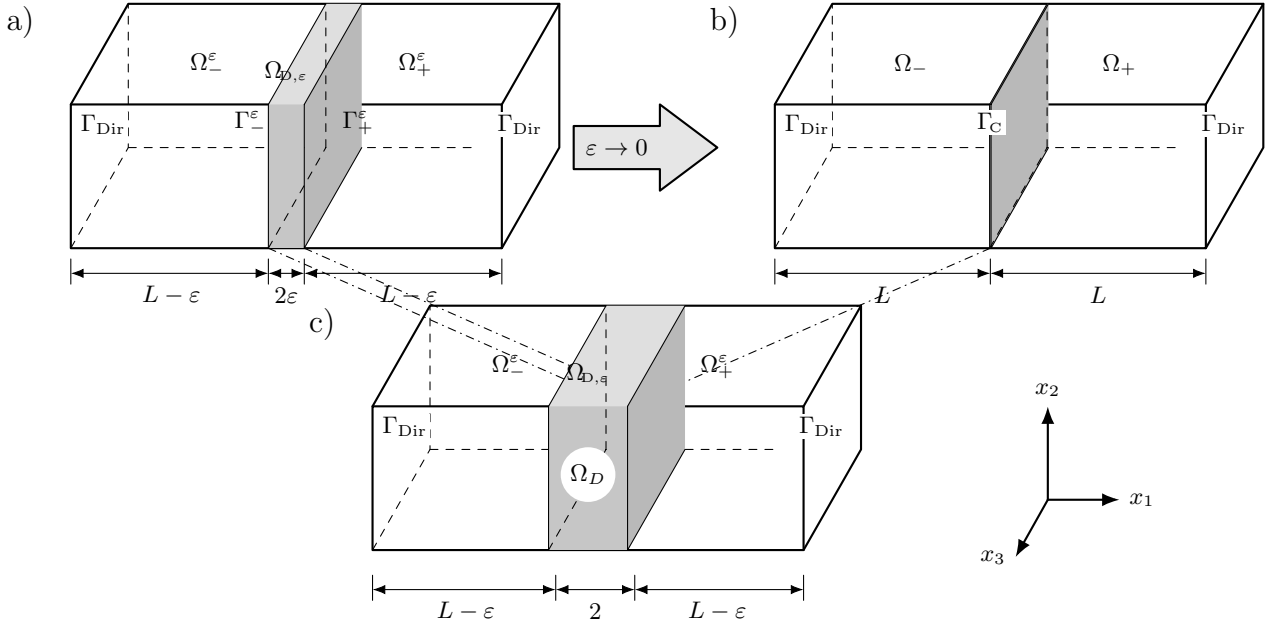
\begin{figure}
 \begin{tikzpicture}[
    scale=0.95,
    line join=round,
    line cap=round,
    >=Latex,
    every node/.style={font=\small}
]


\def\W{6.0}      
\def\H{2.0}      
\def\dx{0.8}     
\def\dy{1.4}     



\newcommand{\DrawCuboid}[8]{%

\begin{scope}[shift={({#1},{#2})}]


\coordinate (A) at (0,0);
\coordinate (B) at (\W,0);
\coordinate (C) at (\W,\H);
\coordinate (D) at (0,\H);

\coordinate (A2) at (\dx,\dy);
\coordinate (B2) at ($(B)+(\dx,\dy)$);
\coordinate (C2) at ($(C)+(\dx,\dy)$);
\coordinate (D2) at ($(D)+(\dx,\dy)$);


\draw[thick]
(A)--(B)--(C)--(D)--cycle;

\draw[thick]
(D)--(D2)--(C2)--(C);

\draw[thick]
(B)--(B2)--(C2);

\draw[dashed]
(A)--(A2)--(D2);

\draw[dashed]
(A2)--($(A2)+(4.8,0)$);


\pgfmathsetmacro{\xL}{#3-#4/2}
\pgfmathsetmacro{\xR}{#3+#4/2}

\coordinate (E1) at (\xL,0);
\coordinate (E2) at (\xR,0);

\coordinate (F1) at (\xL,\H);
\coordinate (F2) at (\xR,\H);

\coordinate (E1b) at ($(E1)+(\dx,\dy)$);
\coordinate (E2b) at ($(E2)+(\dx,\dy)$);

\coordinate (F1b) at ($(F1)+(\dx,\dy)$);
\coordinate (F2b) at ($(F2)+(\dx,\dy)$);

\fill[gray!45]
(E1)--(E2)--(F2)--(F1)--cycle;

\fill[gray!25]
(F1)--(F2)--(F2b)--(F1b)--cycle;

\fill[gray!55]
(E2)--(F2)--(F2b)--(E2b)--cycle;

\draw
(E1)--(F1)
(E2)--(F2)
(F1)--(F1b)
(F2)--(F2b)
(E2)--(E2b);

\draw[dashed]
(E1)--(E1b)
(E1b)--(F1b);


\node at (1.9,2.55) {$#5$};
\node at (4.8,2.55) {$#6$};

\node at (#3,2.45) {$#7$};

\node[fill=white,inner sep=1pt]
at (0.45,1.72)
{$\Gamma_{\mathrm{Dir}}$};

\node[fill=white,inner sep=1pt]
at (6.25,1.72)
{$\Gamma_{\mathrm{Dir}}$};

\node[font=\Large]
at (-0.7,3.15)
{#8};

\end{scope}
}


\DrawCuboid
{0}{4.2}
{3.0}{0.5}
{\Omega^\varepsilon_{-}}
{\Omega^\varepsilon_{+}}
{\ODeps}
{a)}


\node at (2.5,5.95) {$\Gamma^\varepsilon_{-}$};
\node at (4.0,5.95) {$\Gamma^\varepsilon_{+}$};




\draw[<->] (0,3.85)--(2.75,3.85);
\node at (1.4,3.55) {$L-\varepsilon$};

\draw[<->] (2.75,3.85)--(3.25,3.85);
\node at (3.0,3.55) {$2\varepsilon$};

\draw[<->] (3.25,3.85)--(6.0,3.85);
\node at (4.7,3.55) {$L-\varepsilon$};

\foreach \x in {0,2.75,3.25,6}
{
    \draw (\x,4.05)--(\x,3.7);
}


\draw[
    thick,
    fill=gray!20
]
(7.0,5.9)
--(8.2,5.9)
--(8.2,6.2)
--(9.0,5.6)
--(8.2,5.0)
--(8.2,5.3)
--(7.0,5.3)
--cycle;

\node at (7.6,5.6) {$\varepsilon \to 0$};


\DrawCuboid
{9.8}{4.2}
{3.0}{0.02}
{\Omega_{-}}
{\Omega_{+}}
{}
{b)}

\node[fill=white,inner sep=1pt]
at (12.8,5.95)
{$\GC$};




\draw[<->] (9.8,3.85)--(12.8,3.85);
\node at (11.3,3.55) {$L$};

\draw[<->] (12.8,3.85)--(15.8,3.85);
\node at (14.3,3.55) {$L$};

\foreach \x in {9.8,12.8,15.8}
{
    \draw (\x,4.05)--(\x,3.7);
}


\draw[dash dot]
(2.75,4.2)--(7.5,2.1);

\draw[dash dot]
(3.25,4.2)--(8.0,2.1);

\draw[dash dot]
(12.8,4.2)--(8.0,2.1);


\DrawCuboid
{4.2}{0}
{3.0}{1.1}
{\Omega^\varepsilon_{-}}
{\Omega^\varepsilon_{+}}
{\ODeps}
{c)}

\node[fill=white,circle,inner sep=2pt]
at (7.2,1.05)
{$\Omega_D$};




\draw[<->] (4.2,-0.55)--(6.75,-0.55);
\node at (5.5,-0.85) {$L-\varepsilon$};

\draw[<->] (6.75,-0.55)--(7.85,-0.55);
\node at (7.3,-0.85) {$2$};

\draw[<->] (7.85,-0.55)--(10.2,-0.55);
\node at (9.0,-0.85) {$L-\varepsilon$};

\foreach \x in {4.2,6.75,7.85,10.2}
{
    \draw (\x,-0.35)--(\x,-0.7);
}


\begin{scope}[shift={(13.6,0.7)}]

\draw[->,thick] (0,0)--(1.0,0)
node[right] {$x_1$};

\draw[->,thick] (0,0)--(0,1.3)
node[above] {$x_2$};

\draw[->,thick] (0,0)--(-0.45,-0.8)
node[left] {$x_3$};

\end{scope}

\end{tikzpicture}
\caption{We highlight here the geometry of our domain. We assume (see  Fig.\ a)) that $\Omegaep$ and $\Omegaem$ are the reference configurations of two elastic adherents, bonded by a thin adhesive layer of thickness $\eps$. After a rescaling, we map (see Fig.\ b)) the intermediate layer to the $\eps$-independent configuration $\Omega_D$. As the thickness $\eps$ converges to zero, we recover the structure depicted in Fig.\ c), where the intermediate adhesive layer is replaced by the interface $\GC$.   The   tikz code for this picture was created with the help of AI. }
\label{fig2}
\end{figure}

 More precisely, (see Figure \ref{fig2}), we have 
 \begin{subequations}
 \label{epsilon-domains}
 \begin{align}
  \label{epsilon-domains-1}
  & 
 \Omega = \Omega_{-}^{\eps} 
\cup \OD^\eps \cup \Omega_{+}^{\eps}
\qquad \text{with}
\\
 &
  \label{epsilon-domains-2}
\begin{cases}
\Omega_{-}^\eps = (-L,-\eps) \times \GC,
\\
\OD^\eps = (-\eps,\eps)\times \GC,
\\
\Omega_{+}^\eps = (\eps,L) \times \GC,
\end{cases}
\text{  the interfaces } 
\Gamma_{\pm \eps} = \{ \pm \eps \} \times \GC\,, \qquad \text{and  } \eps \down 0.
\end{align}
\end{subequations}
\par
The Energetic formulation for  the damage system \eqref{unrescaled} in the domains 
$ \Omega_{\pm}^{\eps} $ and 
$ \OD^\eps$
 is  codified by
 \begin{itemize}
 \item[-] the state spaces $ \Xeps :=  \mathrm{H}^1(\Omega;\R^d) \ti  \Hs (\OD^\eps) $
 \item[-] 
the dissipation potentials
\begin{equation}
\label{dissReps}
 \mathcal{R}_\eps: \mathrm{L}^1(\OD^\eps) \to [0,\infty], \qquad 
\mathcal{R}_\eps (v) = \int_{\OD^\eps} \frac1\eps\mathrm{R}(v(x)) \dd x \,,
\end{equation}
 \item[-] 
 the energy functionals
$\calE_\eps : [0,T]     \ti \mathrm{H}^1(\Omega;\R^d) \ti  \Hs (\OD^\eps) \to \R$ 
\begin{subequations}
\label{Eeps}
\begin{equation}
\label{Eeps-1}
 \calE_\eps(t,\uu,z) : =\mathcal{Q}_{\mathrm{el}}^{\eps}(t,\uu, z)+ \calF^\eps(z),
 \end{equation}
where the 
elastic energy functionals $\mathcal{Q}_{\mathrm{el}}^\eps : [0,T]\ti \mathrm{H}^1(\Omega;\R^d) \ti  \mathrm{L}^\infty (\OD^\eps) \to \R$ are given  by
\begin{equation}
\label{Eeps-2}
\begin{gathered}
\mathcal{Q}_{\mathrm{el}}^\eps(t,\uu,z) : = \int_{\Omegaem {\cup} \Omegaep  }  \rmW(\nabla \uu) \dd x + \int_{\OD^\eps}  \WD(z,\nabla \uu) \dd x - \langle \effe(t), \uu \rangle_{\mathrm{H}^1(\Omega;\R^d)} 
\\
 \text{with} \quad  \begin{cases}
\rmW & \text{from \eqref{Wgeneral}},
\\
\WD(z,F)= \frac{ \mu\EEE}{2}\gr^\eps (z) \mathbb{I} F{:}F\,,  & \mu>0\EEE \text{ and } \gr^\eps (z): = z^2 +\eps\,,
\end{cases}
\end{gathered}
\end{equation}
whereas the potential energies $\calF^\eps: \Hs(\OD^\eps) \to \R{\cup}\{\infty\}$ 
\begin{equation}
\label{Eeps-3}
  \calF^\eps (z) =   \eps^\rho \int_{\OD^\eps}  
 \frac1{q\eps} \EEE |\nabla z|^q   \dd x  +\EEE  \int_{\OD^\eps}  
\left[ \frac1{2\eps}|\nabla z|^2{+} 
 \frac1\eps\upphi (z) \right]   \dd x\,
 \end{equation}
 feature suitable rescaling factors, with the exponent $\rho>0$ specified later  (see also Remark \ref{rk:rho}). \EEE
 Let us merely point out, here, that the rescaling factor 
$\tfrac1{\eps}$ in front of the two gradient terms reflects the fact that they are defined by an integration on $\OD^\eps$.   \end{subequations}
  \end{itemize}
 For the external loading $\effe$ we now assume that property \eqref{force0} is valid uniformly for all $\eps>0;$ more precisely we claim that
\begin{equation}
\label{force}
\effe \in \mathrm{C}^1([0,T];\mathrm{H}^1(\Omega;\R^d)^*) \; \text{ with } \ \mathrm{supp} (\effe(t)) \cap  \bigcup_{\eps>0} \overline{\Omega^\eps_{\!\scriptscriptstyle{\rm D}}} = \emptyset   \text{ in the sense of distributions for all } t \in [0,T].
\end{equation}
\EEE
  \par
 Theorem \ref{thm:Marita} guarantees that   for every $\eps>0$ there exists an Energetic solution $(\uu_\eps, z_\eps)$ to the rate-independent system $(\Xeps, \calE_\eps, \calR_\eps)$. 

\begin{remark}
\label{on-greps}
\slshape
 Although the choice for the coefficient function $\gr^\eps$ could be slightly extended, in our analysis \EEE 
it is crucial that 
\begin{equation}
\label{gamma-eps-coerc}
 \gr^\eps (z) \geq \eps \qquad \forall\, z \in [0,1].
\end{equation}
In fact, 
 the scaling regime considered in this paper  is precisely individuated by 
\eqref{gamma-eps-coerc}, which  would  correspond, in the setup of \cite{MiRoTh10DDNE},   to the constraint 
\begin{equation}
\label{constraint-Marita}
 z\in [\eps^\gamma, 1], \qquad\text{with  the choice $\gamma=1$.}
\end{equation}
This analogy resides in the fact that,   in \cite{MiRoTh10DDNE}  the elastic energy density $\WD$ for the damageable domains corresponds to  $\WD(z,\nabla \uu) = z|\nabla \uu|^p$
(in that, the elastic energy density from \cite{MiRoTh10DDNE}  indeed involves the linearized strain tensor $e(\uu)$ in place of $\nabla \uu$).
We emphasize that in \cite{MiRoTh10DDNE} 
the exponents $p>1$ and  $\gamma$ either fulfill $\gamma>p-1$, or $0<\gamma<p-1$, whereas the case
$\gamma=p-1$ is  not treated in  \cite{MiRoTh10DDNE}. Hence, for an elastic energy density with quadratic (i.e., $p=2$) growth, \eqref{constraint-Marita} (or, equivalently, \eqref{gamma-eps-coerc}) singles out the case unexplored in  \cite{MiRoTh10DDNE}, which we propose to analyze in what follows.
\end{remark} \EEE 

 \subsection{Rescaling for dimension reduction}
 \label{subs:rescaling}
 Along the footsteps of \cite{MiRoTh10DDNE}, first of all  we   will set the damage flow rule  in a \emph{fixed} domain 
 (see Fig.\ \ref{fig2}c). \EEE  To this end, we introduce
 \begin{itemize}
 \item[--] the rescaled damage domain
 $
 \OD= (-1,1) \ti \GC \quad \text{with } \res{x}= (\res{x}_1, x') \in \OD
 $
 and the short-hand  $x'= (x_2,\cdots, x_d)$; \EEE
 \item[--] the rescaling mapping 
 \begin{equation}
 \label{def-resc}
 \resmap: \OD \to \OD^\eps, \qquad  \resmap(x_1,x') = (\eps x_1, x')
 \end{equation}
 which induces 
 the operator
 \begin{equation}
 \label{resop-def}
 \resop: \mathrm{L}^1(\OD^\eps) \to \mathrm{L}^1(\OD) \qquad z\mapsto \res{z}: = z{\circ}\resop\,.
 \end{equation}
 \end{itemize}
 Using $\resop$, we may rescale the elastic energy $ \calQ_{\mathrm{el}}^\eps$ in such a way as to obtain a functional  over the \emph{fixed} function space $[0,T]\ti \mathrm{H}^1(\Omega;\R^d)\ti \mathrm{L}^\infty (\OD) $. Namely, we define 
  $\res{Q}_{\mathrm{el}}^\eps: [0,T]\ti \mathrm{H}^1(\Omega;\R^d)\ti \mathrm{L}^\infty (\OD) \to \R$ as
 \begin{equation}
 \label{eq:resQep}
 \res{Q}_{\mathrm{el}}^\eps(t,\uu,\res{z}) : = \calQ_{\mathrm{el}}^\eps(t,\uu,\resop^{-1}(\res{z})) =   \int_{\Omega_-^{\eps} {\cup} \Omega_+^{\eps} }  \rmW(\nabla \uu) \dd x + \int_{\OD^{\eps}}  \WD(\res{z} {\circ} \resmap^{-1},\nabla\uu) \dd x - \langle \effe(t), \uu \rangle_{\mathrm{H}^1(\Omega;\R^d)}\,.
\end{equation}
Let us emphasize that, here, along the footsteps of \cite{MiRoTh10DDNE}  we do not rescale the undamageable domains
$\Omega_{\pm}^\eps$ and, accordingly, do not rescale the displacement variable. 
\par
Instead, 
by our choice of the rescaling factor $\frac1\eps$ in \eqref{Eeps-3},  the different  contributions to  $\calF^\eps$
 transform in the following way:
 \begin{subequations}
 \label{damageable-rescaled}
 \begin{align}
 &
 \label{damageable-rescaled-1}
  \int_{\OD^\eps}  
\frac1\eps\upphi (z)  \dd x=  \int_{\OD}  
\upphi (\res{z})  \dd \res{x} \,,
\\
 &
 \label{damageable-rescaled-1/2}
  \int_{\OD^\eps}  
\frac1{2\eps} |\nabla z|^2  \dd x=   \frac12\res{A}_2^\eps(\res{z}) \qquad \text{with }  \res{A}_2^\eps(\res z) : = \int_{\OD} \left( \left| \tfrac1\eps \partial_{\res{x}_1} \res z\right|^2  {+} |
\partial_{x'} \res z|^2 \right) \dd \res{x}_1 \dd x'
\\
&
 \label{damageable-rescaled-2}
\begin{aligned}
&
\int_{\OD^\eps}  
 \frac1{q\eps}  \EEE |\nabla z|^q  \dd x=    \frac1q \EEE \res{A}_{q}^\eps(\res{z}) \qquad \text{with }  \res{A}_{q}^\eps(\res z) : = \int_{\OD} \left( \left| \tfrac1\eps \partial_{\res{x}_1} \res z\right|^2  {+} |
\partial_{x'} \res z|^2 \right)^{\frac{q}{2}} \dd \res{x}_1 \dd x'
  \end{aligned}
 \end{align}
 \end{subequations}
with the notation  $ \partial_{\res{x}_1}  $  for the partial derivative w.r.t.\ the first variable, and $\nabla_{x'}$ for the gradient w.r.t.\  $x'= (x_2,\cdots, x_d)$.
 In what follows, for simplicity we will 
simply write $\res{A}^\eps$ in place of $\res{A}_2^\eps$. 
 \EEE 

All in all,
 the rescaled driving energy $ \res{E}_\eps: [0,T]\ti \mathrm{H}^1(\Omega;\R^d) \ti \Hs(\OD) \to \R{\cup}\{\infty\}$ is given by 
\begin{equation}
\label{resE-eps}
 \res{E}_\eps(t,\uu, \res{z}) : =  \res{Q}_{\mathrm{el}}^\eps(t,\uu,\res{z})+\res{F}_\eps (\res{z}) 
 \qquad \text{with } \res{F}_\eps (\res{z})  = 
  \frac{\eps^\rho}q \EEE  \res{A}_{q}^\eps(\res{z})  +  \frac12\res{A}^\eps(\res{z}) 
+ \int_{\OD}  
\upphi (\res z)  \dd \res{x} \,.
\end{equation}
In the same way, using $\resop$ we define $ \res{R}_\eps: \mathrm{L}^1(\OD) \to [0,\infty]$
 \begin{equation}
 \label{resZ-eps}
 \res{R}_\eps(\res{v}) : = \calR_\eps(\resop^{-1}(\res{v})) =  \int_{\OD^\eps} \tfrac1\eps \mathrm{R}(\res{v}{\circ} \resmap^{-1}) \dd x =  \int_{\OD} 
 \mathrm{R}(\res{v}) \dd \res{x}_1 \dd x' 
  \,.
\end{equation}
In fact, $ \res{R}_\eps$ does not depend on $\eps>0$, and will thus be simply denoted as $\res R$. 
The `rescaled' damaged system is thus individuated by the triple
 \begin{equation}
 \label{triple-notation}
 (\resXeps,  \res{E}_\eps, \res{R}) \qquad \text{with } \resXeps :=  \mathrm{H}^1(\Omega;\R^d) \ti  \Hs (\OD).
 \end{equation}
 Recall that, 
for every $\eps>0$,
 Theorem \ref{thm:Marita} guarantees that  there exists an Energetic solution $(\uu_\eps, z_\eps): [0,T] \to \Xeps$ to the rate-independent system
 $(\Xeps, \calE_\eps, \calR_\eps)$. Then, the pair
 $(\uu_\eps,\res{z}_\eps: = z_\eps{\circ}\resop)$,
 with 
 \begin{equation}
 \label{rescaled-sol-reg}
 \uu_\eps \in \mathrm{L}^\infty (0,T;  \mathrm{H}^1(\Omega;\R^d) ), \qquad \res{z}_\eps \in \mathrm{L}^\infty (0,T;  \Hs (\OD)) \cap \mathrm{BV}([0,T]; \mathrm{L}^1(\OD)),
 \end{equation}
 provides an Energetic solution to the rate-independent system  
$(\resXeps,  \res{E}_\eps, \res{R})$,   fulfilling 
\begin{subequations}
\label{enform-dam-rescaled}
\begin{itemize}
\item[--] the \emph{global stability condition}
\begin{equation}
\label{glSTAB-reseps}
\res{E}_\eps(t,\uu(t), \res{z}(t)) \leq \res{E}_\eps(t,\hat{\uu}, \hat{\res{z}})  + \res{R}(\hat{\res{z}}{-} \res{z}(t)) \qquad \text{for all } (\hat{\uu},   \hat{\res{z}})  \in  \resXeps;
\end{equation}
\item[--] the  \emph{global energy balance}
\begin{equation}
\label{glEnBal-reseps}
\res{E}_\eps(t,\uu(t), \res{z}(t)) +\mathrm{Var}_{\res{R}}(\res{z};[s,t])  = \res{E}_\eps(s,\uu(s), \res{z}(s))+ \int_s^t \partial_t \res{E}_\eps(r,\uu(r), \res{z}(r)) \dd r  \qquad \text{for all } [s,t]\subset [0,T]\,.
\end{equation}
\end{itemize}
\end{subequations}
We will denote by 
\[
\boldsymbol{\mathfrak{S}}_\eps: =\bigcup_{t\in [0,T]} \mathfrak{S}_\eps(t)
\]
the related stable set.
\subsection{The rate-independent system for cohesive delamination}
 The main result of the paper guarantees the  \emph{Evolutionary $\Gamma$-convergence} of systems  $(\resXeps,  \res{E}_\eps, \res{R})_\eps$ to a rate-independent system for cohesive delamination introduced below. 
 \par
 First of all, we  recall  the notation 
\begin{equation}
\label{limit-domains}
\Omega_- = (-L,0) \times \GC, \qquad \Omega_+ = (0,L) \times \GC, 
\end{equation}
 see Fig.\ \ref{fig2}b. \EEE Let us now settle the functional setup.  For   the displacement variable,  we will consider the function space
\[
\rmH_{\GDir}^1(\Omega_{-}{\cup}\Omega_+) : = \{ \uu \in \mathrm{H}^1(\Omega_{-}{\cup}\Omega_+) \, : \ \uu = 0 \text{ on } \GDir\}\,.
\]
Following  \cite{MiRoTh10DDNE},  we introduce the following notion of  weak convergence:  given $(\uu_n)_n,\, \uu \in \mathrm{H}^1(\Omega_-{\cup} \Omega_+;\R^d)$, we write
\begin{equation}
\label{Tau_u_cvg}
 \uu_n \TU \uu \quad  \text{ if and only if } \quad 
\uu_n \weakto \uu  \text{ in }  \mathrm{H}^1(\Omega_-^\nu{\cup} \Omega_+^\nu;\R^d) \text{ for all } \nu \in (0,\eta].
\end{equation}
 Observe that this convergence notion is in fact equivalent to convergence in $\mathrm{H}_{\mathrm{loc}}^1(\Omega_-{\cup} \Omega^+;\R^d)$.

The state space for the delamination variable $\res{z}$ shall encompass the information that $\res{z}$ does not depend on the variable $\res{x}_1$
\[
\res{Z} : = \{ \res{z} \in \mathrm{H}^1(\OD) \, : \ \partial_{\res{x}_1} \res{z}= 0\}\,,
\]
and reflects the loss of spatial regularity (from $\Hs (\OD)$, to $\mathrm{H}^1(\OD)$) due to the vanishing coefficient $\eps^\rho$ modulating the 
 contribution  $\res{A}_q^\eps$ \EEE  in 
\eqref{resE-eps}. 
\par
 The rate-independent system 
 obtained in the limit $\eps\down 0$  is codified by 
 \begin{itemize}
 \item[--]  the state space
 \begin{equation}
 \label{spaceX-lim}
 \res{X}=  \rmH_{\GDir}^1(\Omega_-{\cup} \Omega_+;\R^d)\ti \res{Z};
 \end{equation}
 in what follows, we will consider on $\res{X}$ the topology associated with the following convergence notion: for $(\uu_n,\res{z}_n) , \, (\uu,\res{z})\in \res{X}$, 
 \begin{equation}
\label{wk-cvg-resX}
(\uu_n,\res{z}_n) \weaktoX (\uu,\res{z}) \text{ in } \res{X}  \qquad\text{if and only if}\qquad \begin{cases}
\uu_n \TU \uu & \text{in } \mathrm{H}^1(\Omega_-{\cup} \Omega_+;\R^d),
\\
\res{z}_n \weakto \res{z} & \text{in } \mathrm{H}^1(\OD);
\end{cases}
 \end{equation}
  \item[--]  the driving energy functional $\res{E}:[0,T]\ti \res{X} \to   \R{\cup}\{\infty\}$
  \begin{subequations}
   \label{En-lim}
   \begin{equation}
 \label{En-lim-1}
 \res{E}(t,\uu, \res{z}) : =   \int_{\Omega_- {\cup} \Omega_+ }  \rmW(\nabla \uu) \dd x +\res{G}(\uu,\res{z}) +\res{F} (\res{z}) - \langle \effe(t), \uu \rangle_{\mathrm{H}^1(\Omega;\R^d)} 
\end{equation}
with the coupling energy
    \begin{equation}
 \label{En-lim-2}
 \res{G} :  \rmH_{\GDir}^1(\Omega_{-}{\cup}\Omega_+){\times}  \mathrm{L}^1(\OD) \to [0,+\infty), \qquad \res{G}(\uu, \res{z}): = \int_{\GC} \left (I_{\{\boldsymbol{0}\}}(\res{z} \JUMP{\uu}){+}  \frac\mu2 \EEE |\JUMP{\uu}|^2\right) \dd x'
 \end{equation}
 (recall that, for $\vv \in \R^d$
 we have
$I_{\{\boldsymbol{0}\}}(\vv) = 0 $ if $\vv = \boldsymbol{0}$, and $+\infty$ otherwise),
and 
    \begin{equation}
 \label{En-lim-3}
 \res{F} : \res{Z} \to  \R{\cup}\{\infty\}, \qquad 
 \res{F} (\res{z})  := 
  \int_{\OD}  \left[ \frac12 |\nabla \res{z}|^2{+}
\upphi (\res z) \right]  \dd x 
\end{equation}
\end{subequations}
(since $\res z \in \res{Z}$, we obviously have $ |\nabla \res{z}|^2 =  |\nabla_{x'} \res{z}|^2$).  We stress that, with a slight abuse of notation, in \eqref{En-lim-2} we have denoted by $\res{z}$ the trace of the map on $\GC$.\EEE
    \item[--]  the dissipation potential $\res{R}: 
   \mathrm{L}^1(\OD) \to [0,\infty]$ given by \eqref{resZ-eps}. 
 \end{itemize}
 \noindent
 Accordingly, 
 a pair $(\uu,\res{z}): [0,T] \to \res{X}$ is an Energetic solution to the rate-independent system  $(\res{X},  \res{E}, \res{R})$  if it complies with 
 \begin{subequations}
\label{enform-delam}
\begin{itemize}
\item[--] the \emph{global stability condition}
\begin{equation}
\label{glSTAB-reseps-bis}
\res{E}(t,\uu(t), \res{z}(t)) \leq \res{E}(t,\hat{\uu}, \hat{\res{z}})  + \res{R}(\hat{\res{z}}{-} \res{z}(t)) \qquad \text{for all } (\hat{\uu},   \hat{\res{z}})  \in  \res{X};
\end{equation}
\item[--] the  \emph{global energy balance}
\begin{equation}
\label{glEnBal-reseps-bis}
\res{E}(t,\uu(t), \res{z}(t)) +\mathrm{Var}_{\res{R}}(\res{z};[s,t])  = \res{E}(s,\uu(s), \res{z}(s))+ \int_s^t \partial_t \res{E}(r,\uu(r), \res{z}(r)) \dd r  \qquad \text{for all } [s,t]\subset [0,T]\,.
\end{equation}
\end{itemize}
\end{subequations}
We will denote by 
\[
\boldsymbol{\mathfrak{S}}: =\bigcup_{t\in [0,T]} \mathfrak{S}(t)
\]
the related stable set.

\subsection{Evolutionary $\Gamma$-convergence}
 We are now in a position to state the 
 main result of this paper, providing 
 the Evolutionary $\Gamma$-convergence, as the thickness $\eps\down 0$, of the rate-independent damage systems $(\Xeps,  \res{E}_\eps, \res{R})$ with well-prepared  initial data $(\uu_\eps^0, \res{z}_\eps^0)$ (where $ \res{z}_\eps^0 = z_\eps^0{\circ} \mathbb{T}_\eps$), to the rate-independent delamination system $(\res{X},  \res{E}, \res{R})$.

We point out that the assumption here below on the exponent $\rho$ in \eqref{Eeps-3} will be instrumental for Lemma \ref{l:liminf-crucial}, which in turn will be needed to establish an asymptotic lower bound for the elastic energies. 
 \begin{theorem}
 \label{thm:main}
Assume that $0<\rho< 2-\tfrac4\lambda$,
 with $\lambda = \min\{4, q\}$. 
  \EEE For each  $\eps \in (0,\eta]$, let $(\uu_\eps, \res{z}_\eps) \in \mathrm{L}^\infty(0,T;\Xeps)$ be an Energetic solution to the rate-independent damage systems $(\Xeps,  \res{E}_\eps, \res{R})$, emanating from the  initial data $(\uu_\eps^0, \res{z}_\eps^0)$. Suppose that there exists $(\uu_0,\res{z}_0) \in \res{X}$ such that,  along a null sequence $\eps_n\downarrow 0$ , the following convergences hold as $n\to\infty$:
\begin{equation}
\label{well-prepared-init-data}
\begin{cases}
(\uu_{\eps_n}^0, \res{z}_{\eps_n}^0)  \weaktoX ( \uu_0,  \res{z}_0)   & \text{in } \res{X}, 
\\
 \res{E}_{\eps_n}(0,\uu_{\eps_n}^0, \res{z}_{\eps_n}^0  ) \to  \res{E}(0,\uu_0, \res{z}_0). & 
 \end{cases}
\end{equation}
Then, there exist a pair $(\uu,\res{z}): [0,T] \to \res{X}$, with 
\begin{equation}
\label{limit-pair}
\uu \in \mathrm{L}^\infty(0,T;  \mathrm{H}^1(\Omega_{-}{\cup}\Omega_+;\R^d)  ), \qquad   \res{z} \in \mathrm{L}^\infty(0,T;  \mathrm{H}^1(\OD)) \cap  \BV([0,T];\mathrm{L}^1(\OD)),
\end{equation}
 such that, along a (not relabeled) subsequence, the following convergences hold as $n\to\infty$
 \begin{equation}
 \label{convergences-thm}
( \uu_{\eps_n}(t) ,\res{z}_{\eps_n}(t)) \weaktoX (\uu(t),  \res{z}(t))  \quad \text{in } \res{X}\qquad \text{for all } t \in [0,T], 
 \end{equation}
 and $(\uu,\res{z})$ is an Energetic solution to the rate-independent system $(\res{X},  \res{E}, \res{R})$.
 \end{theorem}

 \paragraph{\bf Roadmap for the proof.}
 We shall resort to the well-known result 
 \cite[Thm.\ 3.1]{MRS06} \EEE 
 and prove the following conditions (below, we use the notation $\resen$, $n\in\N_\infty: = \N{\cup}\{\infty\}$, to indicate both the energies $\res{E}_{\eps_n}$ for $n\in \N$, and  $\res{E}$ for $n=\infty$):
 \begin{itemize}
 \item[-] \textbf{Uniform power control:}
 \begin{equation}
 \label{power-control} 
 \tag{$\mathrm{C}.1$}
 \exists\, C_P>0 \ \forall\, n \in \N_{\infty} \ \forall\, (t,\uu,\res{z}) \in [0,T]\ti \resX \, : \qquad 
 |\partial_t \resen (t,\uu,\res{z})| \leq C_P (1{+} \resen(t,\uu, \res{z});
 \end{equation}
 observe that, as soon as  \eqref{power-control}  holds, we have
 \[
  \exists\, C_P'>0  \forall\, t \in [0,T]\, :  \sup_{t\in [0,T]} |\resen(t, \uu, \res{z}) | \leq C_P' \left(  \resen (t,\uu,\res{z}) {+}1\right).
 \]
 \item[-] \textbf{Compactness of energy sublevels:}
 \begin{equation}
  \tag{$\mathrm{C}.2$}
 \label{comp-en-subl}
 \begin{aligned}
 &
 \forall\, E \in \R \ \forall\, n \in \N_\infty 
 \\
 &
 \text{ the sublevels } L_{E}^n : = \{(\uu,\res{z}) \in \res{X}\, : \  \sup_{t\in [0,T]} |\resen(t, \uu, \res{z}) | \leq E \} \text{ are relatively sequentially compact in } \res{X};
 \end{aligned}
 \end{equation}
 \end{itemize}
 Additionally,  we will verify a series of  `compatibility' conditions, involving $(t_n,\uu_n,\res{z}_n)_n, \, (t,\uu,\res{z}) \in [0,T]{\times}\resX$ 
such that 
\[
t_n \to t,\  (\uu_n, \res{z}_n) \weaktoX (\uu,\res{z}) \text{ in } \resX, \  (\uu_n, \res{z}_n) \in \res{S}_{\eps_n}(t_n) \ \forall\, n \in \N
\]
namely
\begin{itemize}
 \item[-] 
 \textbf{Continuous convergence  of $\partial_t \resen$:}
 \begin{equation}
 \label{cvg-powers}
  \tag{$\mathrm{C}.3$}
\partial_t \resen(t_n,\uu_n,\res{z}_n) \to \partial_t \res{E}(t,\uu,\res{z});
 \end{equation} 
  \item[-] 
 \textbf{Lower $\Gamma$-limit for  $(\resen)_n$:}
 \begin{equation}
 \label{gamma-en}
  \tag{$\mathrm{C}.4$}
 \res{E}(t,\uu,\res{z}) \leq \liminf_{n\to\infty} \resen(t_n,\uu_n,\res{z}_n);
 \end{equation} 
  \item[-] 
 \textbf{Mutual recovery sequence ($\mathrm{MRS}$) condition:}  for all $(\hat{\uu},\hat{\res{z}}) \in \res{X}$ there exists a  
 \emph{mutual recovery sequence} $(\hat{\uu}_n,\hat{\res{z}}_n)_n \subset \res{X} $, with $(\hat{\uu}_n,\hat{\res{z}}_n)\weaktoX (\hat{\uu},\hat{\res{z}})$, such that 
 \begin{equation} 
 \label{MRS}
   \tag{$\mathrm{C}.5$}
 \limsup_{n\to\infty} \left(  \resen(t_n,\hat{\uu}_n,\hat{\res{z}}_n) {+} \res{R}(\hat{\res{z}}_n{-}\res{z}_n) {-}   
 \resen(t_n,\uu_n,\res{z}_n)  \right) \leq  \res{E}(t,\hat{\uu},\hat{\res{z}}) + \res{R}(\hat{\res{z}}{-}\res{z}) -   
 \resen(t,\uu,\res{z})\,. 
 \end{equation}
 \end{itemize}
 As shown in \cite{MRS06},  the $\mathrm{MRS}$ condition in fact guarantees  the \emph{closedness of the stable sets}, in the sense  that, for all sequences  $(t_n,\uu_n,\res{z}_n)_n, \, (t,\uu,\res{z}) \in [0,T]{\times}\resX$, 
 \[
 \begin{cases}
t_n \to t,  \ (\uu_n, \res{z}_n) \weaktoX (\uu,\res{z})   \text{ in } \resX, 
\\
 (\uu_n, \res{z}_n) \in \res{S}_{\eps_n}(t_n)  
\  \forall\, n \in \N,
 \end{cases}
  \ \Longrightarrow \ (\uu,\res{z}) \in \res{S}(t)\,.
\]

 \paragraph{\bf Overview of the proof.} The proof of Theorem \ref{thm:main} will be carried out as follows:
 \begin{itemize}
 \item[--] Conditions \eqref{power-control}--\eqref{cvg-powers} will be verified in Section \ref{s:4};
 \item[--] the lower $\Gamma$-limit condition \eqref{gamma-en} will be proved in Section  \ref{s:5};
 \item[--] the $\mathrm{MRS} $ condition \eqref{MRS} will be established in Section \ref{s:6}, along with the conclusion of the proof. 
 \end{itemize}

\section{Power control/power convergence properties,  coercivity of the energy}
\label{s:4}
\noindent
In this section, we prove the validity of \eqref{power-control}--\eqref{cvg-powers}. All the results in this section and the next one are valid  under the assumptions of Theorem \ref{thm:main}, therefore we will not explicitly repeat them. 

We begin by showing the first and third conditions.
\begin{lemma}
\label{lemma:C1-C3}
There exists $C_P>0$ such that for every $n \in \N_{\infty}$ and for all $(t,\uu,\res{z}) \in [0,T]\ti \resX$ it holds 
 \begin{equation*}
 |\partial_t \resen (t,\uu,\res{z})| \leq C_P (1{+} \resen(t,\uu, \res{z})).
 \end{equation*}
 Further, if \[
t_n \to t,\  (\uu_n, \res{z}_n) \weaktoX (\uu,\res{z}) \text{ in } \resX, \  (\uu_n, \res{z}_n) \in \res{S}_{\eps_n}(t_n) \ \forall\, n \in \N,
\]
then
 \begin{equation*}
 \partial_t \resen(t_n,\uu_n,\res{z}_n) \to \partial_t \res{E}(t,\uu,\res{z}).
 \end{equation*} 
\end{lemma}
\begin{proof}
Recalling the definition of $\res{Q}_{\mathrm{el}}^{\eps}$ in Subsection \ref{subs:rescaling} and \eqref{resE-eps}, by \eqref{partial-control-eta} we find
\begin{equation}
\label{eq:power-computation}
\partial_t\res{E}_{\eps_n}(t,\uu,\res{z})=\partial_t\res{Q}_{\mathrm{el}}^{\eps_n}(t,\uu,\res{z})=-\langle \effe'(t), \uu \rangle_{\mathrm{H}^1(\Omega;\R^d)}.
\end{equation}
Thus,
\begin{align*}
|\partial_t\res{E}_{\eps_n}(t,\uu,\res{z})|\leq \|\effe'\|_{\rmL^{\infty}(0,T;\mathrm{H}^1_{\Gamma_{\mathrm{dir}}}(\Omega;\mathbb{R}^d)^*)}\| \uu \|_{\mathrm{H}^1(\Omega_{-,\eta}{\cup}\Omega_{+,\eta};\R^d)}\,\leq C_P(1{+} \resen(t,\uu, \res{z})),
\end{align*}
where in the latter step, we have used a combination of Poincar\'e and Young inequalities with \eqref{Wgeneral}, and where the constant $C_P>0$ only depends on $\Omega$, $W$, and $\effe$.

The second part of the statement follows by \eqref{eq:power-computation} and \eqref{force}.
\end{proof}


We pin down, for later use, the following estimate 
\begin{lemma}
Let $\uu \in   \mathrm{H}^1(\Omega_-{\cup} \Omega_+;\R^d)$ with traces $\uu_\pm$ on $\Omega_\pm$.  Then, for every $\nu \in (0,\eta]$ there holds
\begin{equation}
\label{trivial-but-useful}
\| \uu ({\pm}\nu, \cdot) - \uu_\pm \|_{\mathrm{L}^2(\GC)} \leq \nu^{1/2} \| \uu \|_{\mathrm{H}^1(\Omega_\pm)}\,. 
\end{equation}
\end{lemma}
\begin{proof}
It suffices for instance to observe that 
\[
\begin{aligned}
\| \uu (\nu, \cdot) - \uu_+ \|_{\mathrm{L}^2(\GC)}^2  & =\int_{\GC} |\uu (\nu,x'){-} \uu_+(x')|^2 \dd x'
\\
& 
= \int_{\GC} \left| \int_0^\nu \partial_{x_1} \uu (s,x') \dd s \right|^2 \dd x'
\\
& 
\leq \nu \iint_{(0,\nu){\times}\GC} |\partial_{x_1} \uu (s,x')|^2 \dd s \dd x' \leq  \nu  \| \uu \|^2_{\mathrm{H}^1 (\Omega_+) \EEE}\,.
\end{aligned}
\]
\end{proof}

\par

We immediately have the following result.
\begin{lemma}[Uniform coercivity properties]
\label{lemma:coercivity}
The following holds
\begin{equation}
\label{uniform-coercivity-properties}
\begin{aligned}
&
\forall\, E>0 \ \ \exists\, C_E>0 \ \ \forall\, \eps \in (0,\eta] 
\ \ \forall\, \nu \in [\eps,\eta] 
 \ \   \forall\, t \in [0,T] \ \forall\, \res{z} \in \mathrm{L}^1(\OD) 
\ \forall\, \uu \in \mathrm{Argmin}_{\vv \in \rmH_{\GDir}^1(\Omega;\R^d)}  \res{Q}_{\mathrm{el}}^\eps(t,\vv,\res{z}) \,: 
\\
&
 \res{E}_\eps(t,\uu,\res{z}) \leq E \ \Longrightarrow \ 
 \| \uu \|_{\mathrm{H}^1(\Omega_-^{\nu}{\cup}\Omega_+^{\nu};\R^d)} 
 +\| [\gr^\eps(\resop^{-1}(\res{z}))]^{1/2}  \nabla\uu\|_{\mathrm{L}^2(\OD^\eps)}
 \\
 & \qquad \qquad \qquad \qquad \qquad \qquad \qquad 
+\eps^\rho  \res{A}_{q}^\eps(\res{z})  \EEE
 + \res{A}^\eps(\res{z})    \leq C_E\,.
 \end{aligned}
\end{equation}
 In particular,   there exists a possibly different constant $C_E'$ such that 
\begin{equation}
\label{ci-salvano}
\eps^{\rho/q} \| \frac1{\eps} \partial_{\res{x}_1} \res{z}\|_{\mathrm{L}^{q}(\OD)}  + \eps^{\rho/q} \| \nabla_{x'} \res{z}\|_{\mathrm{L}^{q}(\OD)} \leq C_E'\,.
\end{equation}
\end{lemma} 
\begin{proof}
First of all, observe that \eqref{Wgeneral} and \eqref{partial-control-eta} yield
\begin{equation}
\label{calc-prelim}
\begin{aligned}
&  \int_{\Omega_-^{\nu} \cup \Omega_+^{\nu} }  \rmW(\nabla\uu) \dd x 
   + \int_{\OD^{\eps}}  \WD(\res{z} {\circ} \resmap^{-1},\nabla\uu) \dd x 
  - \langle \effe(t), \uu \rangle_{\mathrm{H}^1(\Omega;\R^d)}
\\
 & \geq c_{W}  \int_{\Omega_-^{\nu} \cup \Omega_+^{\nu} } |\nabla \uu|^2 \dd x -C_W|\Omega|
  - \|\effe\|_{\mathrm{L}^\infty(0,T;\mathrm{H}^1(\Omega;\R^d)^*) } \| \uu \|_{\mathrm{H}^1(\Omega_{-,\eta}{\cup}\Omega_{+,\eta};\R^d)}
\geq - C\,.
\end{aligned}
\end{equation}
Therefore, from  $E\geq   \res{E}_\eps(t,\uu,\res{z})  $ we gather that 
$E\geq \res{F}_{\eps}(\res{z})$ and,  a fortiori, that $\res z(x) \in [0, 1] \subset [0,1]$ for all $x\in \CLOD$. Thus, taking into account that $\upphi$ is bounded from below by  a linear function, we ultimately have 
\[
\int_{\OD} \upphi(\res z) \dd x \geq -C\,.
\]
Therefore, we infer the following chain of inequalities  for all $\uu \in \mathrm{Argmin}_{\vv \in \mathrm{H}^1(\Omega;\R^d)} \res{Q}_{\mathrm{el}}^\eps(t,\vv,\res{z})  $
and all $ \eps 
\leq \nu \leq \eta$
\[\begin{aligned}
E\geq   \res{E}_\eps(t,\uu,\res{z})  & 
\geq 
  \int_{\Omega_-^{\nu} \cup \Omega_+^{\nu} }  \rmW(\nabla\uu) \dd x 
   + \int_{\OD^{\eps}}  \WD(\res{z} {\circ} \resmap^{-1},\nabla\uu) \dd x 
  - \langle \effe(t), \uu \rangle_{\mathrm{H}^1(\Omega;\R^d)}
  \\
  & \quad 
  +  \frac{\eps^\rho}q \EEE \res{A}_{q}^\eps(\res{z})
   + \frac12\res{A}^\eps(\res{z})  \EEE
  -C
\\
&\geq c_W \| \uu \|_{\mathrm{H}^1(\Omega_-^{\nu} \cup \Omega_+^{\nu} ;\R^d)}^2  
+ \int_{\OD^{\eps}}  \WD(\res{z} {\circ} \resmap^{-1},\nabla\uu) \dd x 
+  \frac{\eps^\rho}q \EEE \res{A}_{q}^\eps(\res{z})
 + \frac12\res{A}^\eps(\res{z})  
\\
& \quad
-C\|\effe\|_{\mathrm{L}^\infty(0,T;\mathrm{H}^1(\Omega;\R^d)^*) }^2  -C,
\end{aligned}
\]
where the last estimate follows from \eqref{calc-prelim}.  Thus, \eqref{uniform-coercivity-properties} and \eqref{ci-salvano} ensue.

\end{proof}

We conclude this section with a result collecting the  compactness properties of sequences with equibounded  energy, cf.\ \cite[Prop.\ 5]{MiRoTh10DDNE}.  In particular, we prove the validity of \eqref{comp-en-subl}.  \EEE
 \begin{proposition}
 Let $(t_\eps,\res{z}_\eps)_\eps\subset [0,T]{\times}\Hs(\OD)$ fulfill
\[
\sup_{\eps \in (0,\eta]}  \res{E}_\eps(t_\eps,\uu_\eps,\res{z}_\eps) \leq E \qquad \text{for some } E>0.
\] 
Let  $(\uu_\eps)_\eps\subset  \rmH_{\GDir}^1(\Omega;\R^d)$ with 
$\uu_\eps \in \mathrm{Argmin}_{\vv \in \rmH_{\GDir}^1(\Omega;\R^d)}  \res{Q}_{\mathrm{el}}^\eps(t_\eps,\vv,\res{z}_\eps)$ for all $\eps \in (0,\eta]$.
\par
 Then, for every null sequence $(\eps_n)_n \subset  (0,\eta]$ there exist a (not relabeled) subsequence,  $\res z \in  \mathrm{H}^1(\OD)$ \EEE 
 and $\uu \in \rmH_{\GDir}^1(\Omega_{-}{\cup}\Omega_+) $ such that, as $n\to\infty$,
 \begin{equation}
 \label{compactness}
 \begin{aligned}
 &
 \uu_{\eps_n} \TU \uu&&  \text{ in } &&  \rmH_{\GDir}^1(\Omega_{-}{\cup}\Omega_+),
 \\
 &
 \res{z}_{\eps_n}\weakto \res z &&  \text{ in } &&  \mathrm{H}^1(\OD)\,, \EEE           
 \end{aligned}
 \end{equation}
 and the sequence 
\begin{equation}
\label{bounded}
(\uu_{\eps_n} )_n \text{ is bounded in $\mathrm{L}^{2^*}(\Omega_- {\cup} \Omega_+;\R^d)$},
\end{equation}
where  $2^*$ is the critical exponent of the Sobolev embedding $\mathrm{H}^1 (\Omega^{\pm}) \subset \mathrm{L}^p(\Omega^{\pm})$, i.e.\ $2^* = \frac{2d}{d-2}$. 
Moreover,
 $0 \leq \res{z}(x)\leq  1$ for almost  all $x \in  \OD$ and $\partial_{\res{x}_1} \res{z} \equiv 0$.
 \end{proposition}
 The \emph{proof} follows by combining the estimates of 
  Lemma \ref{lemma:coercivity} with the arguments from   \cite[Prop.\ 5]{MiRoTh10DDNE}. 

 \section{$\Gamma$-liminf of the energy}
 \label{s:5}
 
This section is devoted to the proof of condition \eqref{gamma-en}. The main difficulty consists in showing that the asymptotic behavior of the bulk contribution on the damage part gives rise, in the relaxation, to a jump of the displacement. This is shown in the result below.



 \begin{proposition}
 \label{prop:limit-jump}
 Assume $0<\rho< 2-\tfrac4\lambda$,
 with $\lambda = \min\{4, q\}$.
  Let $(\uu_{\eps_n},\res{z}_{\eps_n})_n \subset   \rmH_{\GDir}^1(\Omega_{-}{\cup}\Omega_+;\R^d){\ti}  \Hs(\OD)$, \EEE
  $(\uu,\res z) \in  \rmH_{\GDir}^1(\Omega_{-}{\cup}\Omega_+;\R^d){\times} \mathrm{H}^1(\OD)$
 fulfill as $n\to\infty$
 \begin{equation}
 \label{convergences-u/z}
 \uu_{\eps_n} \TU \uu \text{ in } \rmH_{\GDir}^1(\Omega_{-}{\cup}\Omega_+;\R^d)), \qquad \res{z}_{\eps_n}\weakto \res z \text{ in } \mathrm{H}^1(\OD),
 \end{equation}
 with $
\sup_{\eps \in (0,\eta]}  \res{E}_\eps(t_\eps,\uu_{\eps},\res{z}_\eps) \leq E$  for some $ E>0.$
\par
 Then,
 \[
\ \res{G}(\res{z}, \uu) \leq \liminf_{n\to\infty} \int_{\OD^{\eps_n}}  \WD(\res{z}_{\eps_n} {\circ} \mathrm{T}_{\eps_n}^{-1},\nabla\uu_{\eps_n}) \dd x \doteq   L_{\mathrm{bulk}}\,. \EEE 
 \]
  \end{proposition}
 
  


 The proof of Proposition \ref{prop:limit-jump} relies on two intermediate lemmas, which we prove below.
   \begin{lemma}
  \label{l:pedante}
  Assume convergences \eqref{convergences-u/z}. Then,
  \begin{align}
  &
\label{final-conclusion}
 \| \uu_\eps({\pm}\eps,\cdot) - \uu({\pm}\eps,\cdot)  \|_{\mathrm{L}^ r(\GC)}   \longrightarrow 0 \text{ for all } 1 \leq  r \EEE <4\,.
\\
&
  \label{toshow}
  \begin{aligned}
  \lim_{\eps \down 0}   \int_{\GC} \Big[ | &\sqrt{ \gr^\eps(\res z(1,x')) } \uu(\eps,x'){-} \sqrt{\gr^\eps( \res z({-}1,x')) }  \uu({-}\eps,x') |^2 
    \\
  & \quad 
  -  |\sqrt{ \gr^\eps(\res{z}_\eps(1,x') )} \uu_\eps(\eps,x')
  {-} \sqrt{ \gr^\eps(\res{z}_\eps({-}1,x') )}  \uu_\eps({-}\eps,x') |^2 \Big] \dd x'    =0\,.
  \end{aligned}
  \end{align}
  \end{lemma}
  \begin{proof}
  First of all, observe that the functions $\uu({\pm}\eps, \cdot)$, $\uu_\eps ({\pm}\eps, \cdot)$ are well defined as traces
  of $\uu$ and $\uu_\eps$ on the surfaces $\{ {\pm} \eps \} {\times} \GC$. 
  Now, we observe that for almost all $x'\in \GC$ and for \emph{fixed} $\nu\in (0,\eta]$ with $\nu >\eps \down 0$ we have 
  \[
 \uu_\eps(\eps,x') - \uu(\eps,x')  = \uu_\eps(\eps,x') - \uu_\eps(\nu,x') +  \uu_\eps(\nu,x') -  \uu(\nu,x') +    \uu(\nu,x')-  \uu(\eps,x')  \doteq \mu_\eps^1 +\mu_\eps^2 +\mu_\eps^3\,.
  \]
  Now,
 \[
  \mu_\eps^1  = -  \int_{\eps}^{\nu} \partial_{x_1} \uu_\eps (r,x') \dd r
 \]
  Thus,
  \begin{equation}
  \label{miraculous}
  \begin{aligned}
  \|   \mu_\eps^1\|_{\mathrm{L}^2(\GC)}^2 = \int_{\GC} \left|  \int_{\eps}^{\nu} \partial_{x_1} \uu_\eps (r,x') \dd r\right|^2 \dd x' 
 &  \leq (\nu{-}\eps) \iint_{(\eps,\nu){\times}\GC}  |\partial_{x_1} \uu_\eps (r,x') |^2 \dd x' \dd r 
 \\
& \leq  \nu \| \uu_\eps\|_{\mathrm{H}^1(\Omega_\eps^+)}^2 \leq C \nu\,,
  \end{aligned}
  \end{equation}
 where the very last bound follows from \eqref{uniform-coercivity-properties}.
  In the same way, we prove that 
  $
   \|   \mu_\eps^3\|_{\mathrm{L}^2(\GC)}^2  \leq C \nu\,.
  $
Finally, taking into account that, by the definition of   $\TU$ convergence, $\uu_\eps \weakto \uu $ in $\mathrm{H}^1(\Omega_\nu^+)$, by  well-known trace theorems we have that 
that the traces of $\uu_\eps$ on the surface $\{ \nu\}{\times}\GC$ converge to the trace of $\uu$ strongly in $\mathrm{L}^{r}(\GC) $ for all $1\leq r <4$. Therefore, we have that 
\[
 \|   \mu_\eps^2\|_{\mathrm{L}^r(\GC)}\ \longrightarrow 0 \text{ for all } 1 \leq  r \EEE <4\,.
\]
All in all, we have proved that 
\[
\exists\, C>0 \ \ \forall\, \nu \in (0,\eta]\, : \qquad 
\lim_{\eps \down0} \|  \uu_\eps(\eps,\cdot) - \uu(\eps,\cdot) \|_{\mathrm{L}^2(\GC)}^2  \leq 2 C\nu\,.
\]
By the arbitrariness of $\nu$, we ultimately conclude that $ \uu_\eps(\eps,\cdot) - \uu(\eps,\cdot) \to 0$ strongly in $\mathrm{L}^2(\GC)$. On the other hand, the sequence 
$(\uu_\eps(\eps,\cdot) - \uu(\eps,\cdot) )_\eps$ is bounded in $\mathrm{L}^4(\GC)$. Ultimately, we have \eqref{final-conclusion}.

Clearly, the very same convergence holds for $(\uu_\eps({-}\eps,\cdot) - \uu({-}\eps,\cdot) )_\eps$.
\par
Now, let $L$ be the limit in \eqref{toshow}. In what follows, we will use  the place-holders
$\plc_\eps^\pm(\cdot) : = \sqrt{\gr^\eps( \res{z}_\eps({\pm}1,\cdot)) }$, 
$\plcc_\eps^\pm(\cdot) : = \sqrt{ \gr^\eps(\res{z}({\pm}1,\cdot) )}$,  \EEE
$U_\eps^\pm (\cdot) = \uu_\eps({\pm}\eps,x')$, 
$V_\eps^\pm (\cdot) = \uu({\pm}\eps,x')$.  We estimate
\[
\begin{aligned}
|L| &  \leq \lim_{\eps \down 0} \int_{\GC}  |\plcc_\eps^+ V_\eps^+ {-} \plcc_\eps^- V_\eps^-{+} \plc_\eps^+
 U_\eps^+ {-} \plc_\eps^- U_\eps^- | \cdot \left(  |\plcc_\eps^+ V_\eps^+ {-} \plc_\eps^+ U_\eps^+ 
| {+}  |\plcc_\eps^- V_\eps^- {-} \plc_\eps^- U_\eps^- | \right) \dd x  
\\ & \leq \lim_{\eps \down 0} 
\| \plcc_\eps^+ V_\eps^+ {-} \plcc_\eps^- V_\eps^-{+} \plc_\eps^+ U_\eps^+ {-} \plc_\eps^- U_\eps^- \|_{\mathrm{L}^2(\GC)}
\cdot \left( \|\plcc_\eps^+ V_\eps^+ {-} \plc_\eps^+ U_\eps^+ \|_{\mathrm{L}^2(\GC)}{+} \| \plcc_\eps^- V_\eps^- {-} 
\plc_\eps^- U_\eps^-  \|_{\mathrm{L}^2(\GC)}\right)
\\
&
\doteq  \lim_{\eps \down 0}  M_1^\eps (M_{2,+}^\eps{+} M_{2,-}^\eps)
\,.
\end{aligned}
\]
Now, in view of \eqref{uniform-coercivity-properties} and of the bound $(\res{z}_\eps)_\eps \subset [0,1]$, it is immediate to deduce that there exits $C>0$ such that
$
M_1^\eps \leq C\,.
$
On the other hand, we estimate 
\[
M_{2,+}^\eps \leq  \| \plcc_\eps^+ {-} \plc_\eps^+\|_{\mathrm{L}^4(\GC)}^2 \| V_\eps^{+}\|_{\mathrm{L}^4(\GC)}^2 + \| \plc_\eps^+\|_{\mathrm{L}^6(\GC)}^2 
\| V_\eps^{+} {-} U_\eps^+\|_{\mathrm{L}^3(\GC)}^2 \EEE 
\]
and we deduce that $M_{2,+}^\eps \to 0$ as $\eps \down 0$, thanks to \eqref{convergences-u/z}, \eqref{final-conclusion} and to well-known trace theorems. Analogously, we show that $M_{2,-}^\eps \to 0$ as $\eps \down 0$. All in all, we have proved that $L=0$. This concludes the proof.
  \end{proof}

  \begin{lemma}
  \label{l:liminf-crucial} 
  Under the above conditions, we have  
  \begin{equation}
  \label{liminf-crucial}
 L_{\mathrm{surf}}:=  \liminf_{\eps\down 0} \int_{\GC}   \frac{\mu}{2\eps} \EEE  \gr^\eps(\res{z}_\eps(1,x')) |\uu_\eps(\eps,x'){-}\uu_\eps({-}\eps,x') |^2 \dd x'  \leq  L_{\mathrm{bulk}}= \liminf_{\eps 
\down 0}  \int_{\OD^{\eps}}    \WD(\res{z}_{\eps} {\circ} \mathrm{T}_{\eps}^{-1},\nabla\uu_{\eps}) \dd x
  \end{equation}
  and 
  \begin{equation}
  \label{true-lsurf}
 L_{\mathrm{surf}}=  \liminf_{\eps\down 0}  \int_{\GC}    \frac{\mu}{2\eps} \EEE  |\sqrt{\gr^\eps(\res{z}_\eps(1,x'))}\uu_\eps(\eps,x'){-}\sqrt{\gr^\eps(\res{z}_\eps(-1,x'))}\uu_\eps({-}\eps,x') |^2 \dd x' 
  \end{equation}
  \end{lemma}
  In the following proof,
        we will repeatedly use the following estimate
  \[
  (a{+}b)^2 \leq (1{+}\eta) a^2
+ \left(1{+}\frac1\eta \right) b^2 \qquad \text{for all } a,b\in \R \text{ and } \eta>0.
 \]
    \begin{proof}  We start from the right-hand side in \eqref{liminf-crucial}, and observe that, for every fixed $\eps>0$, 
  \begin{align}
  \nonumber
 \int_{\OD^{\eps}}   \WD(\res{z}_{\eps} {\circ} \mathrm{T}_{\eps}^{-1},\nabla\uu_{\eps}) \dd x & = 
 \int_{\OD^{\eps}}  
  \frac{\mu}2 \EEE 
 \gr^\eps( \res{z}_\eps 
\left(\tfrac1\eps r,x' \right) )   \sum_{i,j=1}^d |\partial_{x_i}\uu^j_\eps(r,x'))|^2  \dd r \dd x'\\
&
\label{7.3}
\geq \int_{\OD^{\eps}}   \frac{\mu}2 \EEE 
\gr^\eps( \res{z}_\eps 
\left(\tfrac1\eps r,x' \right) )   \sum_{i=1}^d  |\partial_{x_1}\uu^i_\eps(r,x'))|^2 \EEE   \dd r \dd x':=\Lambda_1.
\end{align}
In order to estimate $\Lambda_1$, we use that 
    \[
  \begin{aligned}
   \partial_{x_1} [\sqrt{ \gr^\eps(z_\eps(r,x')) } \uu_\eps(r,x')]   & = 
  \partial_{x_1} (\sqrt{ \gr^\eps(z_\eps(r,x'))} )  \uu_\eps(r,x') + \sqrt{ \gr^\eps(z_\eps(r,x'))}
\partial_{x_1}  \uu_\eps(r,x')
\\
& =   \partial_{x_1} (\sqrt{ \gr^\eps(\res{z}_\eps
\left(\tfrac1\eps r,x' \right)} )  \uu_\eps(r,x') + \sqrt{ \gr^\eps(\res{z}_\eps
\left(\tfrac1\eps r,x' \right)} )
\partial_{x_1}  \uu_\eps(r,x')\,.
\end{aligned}
  \]
  Therefore, for arbitrary $\eta>0$
  \[
   |  \partial_{x_1} [\sqrt{ \gr^\eps(z_\eps) } \uu_\eps]  |^2   \leq  (1{+}\eta) 
   \gr^\eps(\res{z}_\eps
\left(\tfrac1\eps r,x' \right))
   | \partial_{x_1}  \uu_\eps(r,x') |^2 + \left(1{+}\frac1\eta\right) | \partial_{x_1} [\sqrt{ \gr^\eps(\res{z}_\eps
\left(\tfrac1\eps r,x' \right)}) ]  \uu_\eps(r,x') |^2
  \]
  and thus we have 
  \begin{align}
  \nonumber
   \Lambda_1 & \geq \frac1{1{+}\eta}   \int_{\OD^{\eps}}   \frac\mu2 \EEE  |  \partial_{x_1} [\sqrt{ \gr^\eps(z_\eps) } \uu_\eps]  |^2 \dd x   -   \frac1\eta \EEE 
 \int_{\OD^{\eps}}   \frac\mu2 \EEE | \partial_{x_1} [\sqrt{ \gr^\eps(\res{z}_\eps)
\left(\tfrac1\eps r,x' \right)} ]  \uu_\eps(r,x') |^2
 \dd r \dd x'
\\
\label{7.4}
& \doteq  \frac1{1{+}\eta}  \Lambda_{1,1} -    \frac1\eta \EEE \Lambda_{1,2}     
  \end{align}
  \par
  Now,  by Jensen's inequality, we have 
\[
  \begin{aligned}
 \Lambda_{1,1}  & =   \int_{\GC}\int_{-\eps}^\eps      \frac\mu2 \EEE  |  \partial_{x_1} [\sqrt{ \gr^\eps(z_\eps(r,x')) } \uu_\eps(r,x')]  |^2 \dd r \dd x'  
 \\
& \geq  \int_{\GC}    \frac\mu{2\eps} \EEE   \left|\int_{-\eps}^{\eps}       \partial_{x_1} [\sqrt{ \gr^\eps(z_\eps(r,x')) } \uu_\eps(r,x')] \dd r  \right|^2 \dd x'  \,.
\end{aligned}
\]
  In turn, we observe that 
  \begin{align}
& 
\nonumber
 \int_{\GC}   
  \frac\mu{2\eps} \EEE  \left|\int_{-\eps}^{\eps}   \frac12    \partial_{x_1} (\sqrt{ \gr^\eps(z_\eps(r,x')) } \uu_\eps(r,x')) \dd r  \right|^2 \dd x' 
\\
\nonumber
& =
 \int_{\GC}     \frac\mu{2\eps} \EEE \ |\sqrt{ \gr^\eps(z_\eps(\eps,x')) } \uu_\eps(\eps,x'){-} \sqrt{ \gr^\eps(z_\eps({-}\eps,x')) }  \uu_\eps({-}\eps,x') |^2 \dd x'    
\\
& 
 \label{7.5}
 \stackrel{(1)} {=}  \EEE    
  \int_{\GC}      \frac\mu{2\eps} \EEE \ |\sqrt{ \gr^\eps(\res{z}_\eps(1,x')) } \uu_\eps(\eps,x'){-} \sqrt{ \gr^\eps(\res{z}_\eps({-}1,x')) }  \uu_\eps({-}\eps,x') |^2 \dd x'   \doteq \widehat{\Lambda}_{1,1}
\end{align}
  where  {\footnotesize (1)}  follows from the fact that 
  $ \res{z}_\eps({\pm}1,x')  = (\res{z}_\eps {\circ} T_\eps^{-1})({\pm}\eps,x') \doteq z_\eps ({\pm}\eps,x')$.
  \medskip
  
  \noindent
  \textbf{Claim $1$:} {\sl There holds}
  \begin{equation}
  \label{crucial4liminf}
  \begin{aligned}
 & \liminf_{\eps \down 0}   \int_{\GC}      \frac\mu{2\eps} \EEE \ |\sqrt{ \gr^\eps(\res{z}_\eps(1,x')) } \uu_\eps(\eps,x'){-} \sqrt{ \gr^\eps(\res{z}_\eps({-}1,x')) }  \uu_\eps({-}\eps,x') |^2 \dd x'  
 \\
 &  \geq   \liminf_{\eps \down 0}   \int_{\GC}      \frac\mu{2\eps} \EEE \ |\sqrt{ \gr^\eps(\res{z}_\eps(1,x')) } (\uu_\eps(\eps,x'){-}  \uu_\eps({-}\eps,x')) |^2 \dd x'  \,.
 \end{aligned}
  \end{equation}
  For this, we observe 
  \begin{equation}
  \label{e:develop-square}
  \begin{aligned}
  &
       \int_{\GC}   \frac\mu{2\eps} \EEE \ |\sqrt{ \gr^\eps(\res{z}_\eps(1,x')) } \uu_\eps(\eps,x'){-} \sqrt{ \gr^\eps(\res{z}_\eps({-}1,x')) }  \uu_\eps({-}\eps,x') |^2 \dd x'
      \\
      &
    =    \int_{\GC}     \frac\mu{2\eps} \EEE \ \Big|  \sqrt{ \gr^\eps(\res{z}_\eps(1,x')) }   [\uu_\eps(\eps,x'){-}  \uu_\eps({-}\eps,x') ]
    +[\sqrt{ \gr^\eps(\res{z}_\eps(1,x')) }{-} \sqrt{ \gr^\eps(\res{z}_\eps({-}1,x')) } ]\uu_\eps({-}\eps,x')\Big|^2 \dd x' \doteq I_1+I_2+I_3   
  \end{aligned}
  \end{equation}
  with  
  \[
  \begin{cases}
  I_1 =   \int_{\GC}     \frac\mu{2\eps} \EEE\ |  \sqrt{ \gr^\eps(\res{z}_\eps(1,x')) }   [\uu_\eps(\eps,x'){-}  \uu_\eps({-}\eps,x') ]|^2  \dd x',
  \smallskip
  \\
   I_2 = \int_{\GC}     \frac\mu{2\eps} \EEE \ | [\sqrt{ \gr^\eps(\res{z}_\eps(1,x')) }{-} \sqrt{ \gr^\eps(\res{z}_\eps({-}1,x')) } ]\uu_\eps({-}\eps,x')|^2 \dd x'  
   \end{cases}
   \]
   and $I_3$ the mixed term.
   With Young's inequality we estimate
      \[
      \begin{aligned}
   I_3  & =  \int_{\GC}   \frac\mu{\eps} \EEE  \sqrt{ \gr^\eps(\res{z}_\eps(1,x')) }   [\uu_\eps(\eps,x'){-}  \uu_\eps({-}\eps,x') ] [\sqrt{ \gr^\eps(\res{z}_\eps(1,x')) }{-} \sqrt{ \gr^\eps(\res{z}_\eps({-}1,x')) } ]\uu_\eps({-}\eps,x')  \dd x'  
   \\
   & 
   \geq
   - \beta I_1- \frac1\beta I_2     \qquad \text{for all } \beta>0 \,.
   \end{aligned}
   \]  
   All in all, we have 
   \begin{equation}
   \label{beta-arbitrary}
   \begin{aligned}
    \liminf_{\eps \down 0}  \Lambda_{1,1} \geq  &   
     (1{-}\beta)  \liminf_{\eps \down 0}  \int_{\GC}     \frac\mu{2\eps} \EEE \ |  \sqrt{ \gr^\eps(\res{z}_\eps(1,x')) }   [\uu_\eps(\eps,x'){-}  \uu_\eps({-}\eps,x') ]|^2  \dd x'
    \\
    &
 + \left(1{-}\frac1\beta\right)   \liminf_{\eps \down 0}  \int_{\GC}     \frac\mu{2\eps} \EEE \ | [\sqrt{ \gr^\eps(\res{z}_\eps(1,x')) }{-} \sqrt{ \gr^\eps(\res{z}_\eps({-}1,x')) } ]\uu_\eps({-}\eps,x')|^2 \dd x'  \,.
 \end{aligned}
   \end{equation}
We will now  \EEE 
 prove that
  \begin{equation}
  \label{difference20}
   \lim_{\eps \down 0}   \int_{\GC}      \frac\mu{2\eps} \EEE \ |\sqrt{ \gr^\eps(\res{z}_\eps(1,x')) }{-} \sqrt{ \gr^\eps(\res{z}_\eps({-}1,x')) } |^2 | \uu_\eps(-\eps,x') |^2 \dd x'  =0\,.
  \end{equation}
  Indeed,
  we calculate
  \begin{align}
  &
  \nonumber
 |\sqrt{ \gr^\eps(\res{z}_\eps(1,x')) }{-} \sqrt{ \gr^\eps(\res{z}_\eps({-}1,x')) } |^2 
   \\
   =
    &
      \nonumber
   \frac1{\displaystyle [\sqrt{ \gr^\eps(\res{z}_\eps(1,x')) }{+} \sqrt{ \gr^\eps(\res{z}_\eps({-}1,x')) } ]^2}  | \gr^\eps(\res{z}_\eps(1,x')) {-}  \gr^\eps(\res{z}_\eps({-}1,x')) |^2 
   \\
   \leq
    &
      \nonumber
 \frac1{\displaystyle [\res{z}_\eps(1,x') {+} \res{z}_\eps({-}1,x')]^2}  | \res{z}_\eps^2(1,x') {-}  \res{z}_\eps^2({-}1,x') |^2 
   \intertext{where we have used that $\sqrt{ \gr^\eps(\res{z}_\eps({\pm}1,x'))} \geq  \sqrt{\res{z}_\eps^2({\pm}1,x')} =\res{z}_\eps({\pm}1,x') $}
   = 
    &
      \nonumber
 | \res{z}_\eps(1,x') {-}  \res{z}_\eps({-}1,x') |^2 
  \end{align}
  Therefore,
  \begin{equation}
  \label{quoted-later4null-lim}
  \begin{aligned}
\int_{\GC}        \frac\mu{2\eps} \EEE  |\sqrt{ \gr^\eps(\res{z}_\eps(1,x')) }{-} \sqrt{ \gr^\eps(\res{z}_\eps({-}1,x')) } |^2 | \uu_\eps({-}\eps,x') |^2 \dd x'  
   &  \leq 
       \frac\mu{2\eps} \EEE \int_{\GC}    | \res{z}_\eps(1,x') {-}  \res{z}_\eps({-}1,x') |^2 | \uu_\eps({-}\eps,x') |^2 \dd x'  
 \\
  &  \leq 
          \frac\mu{2\eps} \EEE\| \res{z}_\eps(1,\cdot) {-}  \res{z}_\eps({-}1,\cdot) \|_{\mathrm{L}^4(\GC)}^2  \| \uu_\eps({-}\eps,\cdot)  \|_{\mathrm{L}^4(\GC)}^2
 \end{aligned}
  \end{equation}
 In order to estimate $  \| \res{z}_\eps(1,\cdot) {-}  \res{z}_\eps({-}1,\cdot) \|_{\mathrm{L}^4(\GC)}^2 $, we will use  the bound
  \eqref{ci-salvano}
   for
  $\frac1\eps \partial_{\res{x}_1}\res{z}_\eps$ in $\mathrm{L}^{q}(\OD)$. \EEE 
In the following calculation, we distinguish the cases $q<4$ and $q\geq 4$.
For $q<4$, we have 
\begin{align}
\nonumber
 \| \res{z}_\eps(1,\cdot) {-}  \res{z}_\eps({-}1,\cdot) \|_{\mathrm{L}^4(\GC)}^4 
& = \int_{\GC}
 | \res{z}_\eps(1,x') {-}  \res{z}_\eps({-}1,x') |^{4-q+q} \dd x' \leq 
 \\
 \nonumber
& \leq  2^{4-q}  \int_{\GC}
 | \res{z}_\eps(1,x') {-}  \res{z}_\eps({-}1,x') |^{q} \dd x' 
 \intertext{since $0\leq  \res{z}_\eps({\pm}1,\cdot) \leq 1 $ a.e.\ in $\GC$}
& \nonumber \leq   2^{4-q}    \int_{\GC} \left|\int_{-1}^1  \partial_{\res{x}_1} \res{z}_\eps(\res{x}_1,x')  \dd \res{x}_1 \right|^{q} \dd x' 
 \intertext{and, by Jensen's inequality, we have }
& \nonumber
  \leq 2^{4-q} 2^{q-1}   \int_{\GC} \int_{-1}^1 \left| \partial_{\res{x}_1} \res{z}_\eps(\res{x}_1,x') \right|^{q}   \dd \res{x}_1 \dd x' 
  = C \eps^{q} \left\| \frac1{\eps}  \partial_{\res{x}_1} \res{z}_\eps\right\|_{\mathrm{L}^{q}(\OD)}^{q} 
\end{align}
  In the case $q \geq 4$, we simply have 
\begin{align}
\nonumber
 \| \res{z}_\eps(1,\cdot) {-}  \res{z}_\eps({-}1,\cdot) \|_{\mathrm{L}^4(\GC)}^4 
& \leq C \left(  \| \res{z}_\eps(1,\cdot) {-}  \res{z}_\eps({-}1,\cdot) \|_{\mathrm{L}^{q}(\GC)}^{q}  \right)^{4/q}
 \\
 \nonumber
& 
  \leq C \left(   2^{q-1} \int_{\GC} \int_{-1}^1  \left|\partial_{\res{x}_1} \res{z}_\eps(\res{x}_1,x') \right|^q  \dd \res{x}_1 \dd x'  \right)^{4/q}
\\
& \nonumber
  = C \left(  \eps^{q} \left\| \frac1{\eps}  \partial_{\res{x}_1} \res{z}_\eps\right\|_{\mathrm{L}^{q}(\OD)}^{q}\right)^{4/q} 
  = C \eps^4 \left\| \frac1{\eps}  \partial_{\res{x}_1} \res{z}_\eps\right\|_{\mathrm{L}^{q}(\OD)}^4
\end{align}
All in all, we have shown that 
  \[
   \| \res{z}_\eps(1,\cdot) {-}  \res{z}_\eps({-}1,\cdot) \|_{\mathrm{L}^4(\GC)}^2  \leq C \eps^{\lambda/2} \left\| \frac1{\eps}  \partial_{\res{x}_1} \res{z}_\eps\right\|_{\mathrm{L}^{q}(\OD)}^{\lambda/2} \qquad \text{with } \lambda = \min\{4,q\} \,.
  \]\EEE
 We now recall  \eqref{ci-salvano} and conclude that 
   \begin{equation}
   \label{on-the-verge}
   \| \res{z}_\eps(1,\cdot) {-}  \res{z}_\eps({-}1,\cdot) \|_{\mathrm{L}^4(\GC)}^2 \leq C' \eps^{\tfrac\lambda2}  \eps^{-\tfrac\lambda4 \rho} = C' \eps^{\tfrac\lambda4 (2{-}\rho)}
   \end{equation}
We combine \eqref{on-the-verge} with  \eqref{quoted-later4null-lim}  and use that $  \| \uu_\eps({-}\eps,\cdot)  \|_{\mathrm{L}^4(\GC)}^2 \leq C$
as a consequence of the estimate for  $ \| \uu_\eps  \|^2_{ \mathrm{H}^1(\Omega^\eps_-)\EEE}$,  which in turn follows from the bound \eqref{uniform-coercivity-properties}. 
Therefore, 
    \[
    \int_{\GC}        \frac\mu{2\eps} \EEE  |\sqrt{ \gr^\eps(\res{z}_\eps(1,x')) }{-} \sqrt{ \gr^\eps(\res{z}_\eps({-}1,x')) } |^2 | \uu_\eps({-}\eps,x') |^2 \dd x'  
    \leq C\frac1\eps   \eps^{\tfrac\lambda4 (2{-}\rho)} \longrightarrow 0 
    \]
    since we have required $0<\rho< 2-\tfrac4\lambda$. 
    Thus, \eqref{difference20} follows. Letting $\beta \down 0$ in \eqref{beta-arbitrary}, we obtain the desired
    \eqref{crucial4liminf}. 

     Before proceeding, we observe that constraint $0<\rho< 2-\tfrac4\lambda$, with $\lambda = \min\{4, q\}$, in particular implies $0<\rho<1$. \EEE This latter bound will play a key role in what follows. \EEE

    \medskip
     \textbf{Claim $2$:} {\sl There holds}
  \begin{equation}
  \label{crucial4liminf-2}
  \begin{aligned}
 & \Lambda_{1,2} =
  \int_{\OD^{\eps}}  \frac\mu2 \EEE  | \partial_{x_1} [\sqrt{ \gr^\eps(\res{z}_\eps)
\left(\tfrac1\eps r,x' \right)} ]  \uu_\eps(r,x') |^2  \leq C \eps^{1-\rho} \longrightarrow 0
  \end{aligned}
  \end{equation} 
  {\sl  since $\rho \in ]0,1[$}.   
  \\
  Indeed, 
     we estimate 
   \begin{align}
   \nonumber
     \Lambda_{1,2}  &=    \frac\mu2 \EEE  \int_{\GC}\int_{-\eps}^\eps \left|
   \frac1{ 2\sqrt{ \gr^\eps( \res{z}_\eps(\tfrac1\eps r,x'))}} \partial_{\res{x}_1}\res{z}_\eps(\tfrac1\eps r,x') (\gr^\eps)'(\res{z}_\eps(\tfrac1\eps r,x'))
 \uu_\eps(r,x') \right|^2 \dd r \dd x'
 \\
 &
    \nonumber
  \stackrel{(1)} \leq  \frac\mu2 \EEE  \int_{\GC}\int_{-\eps}^\eps \left|
    \partial_{\res{x}_1}\res{z}_\eps(\tfrac1\eps r,x') 
 \uu_\eps(r,x') \right|^2 \dd r \dd x'
\intertext{where we have used in {\footnotesize (1)} that, for  $\gr^\eps(\zeta) = \zeta^2+\eps$, there holds $\left|\frac{(\gr^\eps)'(\zeta)}{\sqrt{\gr^\eps(\zeta)}} \right| \leq 2$}
 &
    \nonumber
=  \frac{\mu\eps^2}2 \EEE 
\int_{-\eps}^{\eps} \int_{\GC}   \left|\frac 1\eps\partial_{\res{x}_1} \res{z}_\eps \left(\tfrac1\eps r,x' \right) \right|^2
 | \uu_\eps(r,x')|^2 \dd r \dd x'
\\
& \stackrel{(2)}\leq    \frac{\mu\eps^2}2 \EEE    \| \uu_\eps \|_{\mathrm{L}^{2p'}(\OD^\eps)}^2  
\left(  \int_{-\eps}^{\eps}   \int_{\GC}   \left|\frac1\eps \partial_{\res{x}_1} \res{z}_\eps \left(\tfrac1\eps r,x' \right) \right|^{q} \dd r \dd x' \right)^{2/{q}}
\intertext{where in {\footnotesize (2)} we have used H\"older's inequality  with the exponent $p = \tfrac{q}2$ and 
 $p'= \tfrac{q}{q-2}$ \EEE}
&
    \nonumber
  \stackrel{(3)} \leq C  \frac{\eps^2}2 \| \uu_\eps \|_{\mathrm{H}^1(\OD^\eps)}^2   \eps^{-\rho}
  \intertext{in view of estimate \eqref{ci-salvano}, and taking into account that  $q>d$ entails $2p'\leq 2^*$, with $2^*$ the critical exponent for $\mathrm{H}^1(\OD^\eps)$.} 
      \nonumber
\end{align}
Now, in view of \eqref{uniform-coercivity-properties} we have  that 
\[
 \| \uu_\eps \|_{\mathrm{H}^1(\OD^\eps)}^2 \leq C\eps^{-1}\,.
\] 
Therefore, \eqref{crucial4liminf-2} follows.
 
 All in all, we have shown that, for arbitrary $\eta>0$
 \[
   \begin{aligned}
  &   \liminf_{\eps\down 0}  \int_{\OD^{\eps}}    \WD(\res{z}_{\eps} {\circ} \mathrm{T}_{\eps}^{-1},\nabla\uu_{\eps}) \dd x 
    \\
    & \geq 
     \liminf_{\eps\down 0}   \frac1{1+\eta}   \int_{\GC}     \frac{\mu}{2\eps} \EEE \ |\sqrt{ \gr^\eps(\res{z}_\eps(1,x'))} (\uu_\eps(\eps,x'){-}  \uu_\eps({-}\eps,x')) |^2 \dd x' 
       \\
  & \qquad    -  \liminf_{\eps\down 0}      \left(1{+}\frac1\eta\right) C\eps^{1-\rho}
     \end{aligned}
   \]   
   Choosing $\eta = \eps^{(1-\rho)/2}$, so that  the second term on the right-hand side  becomes
   \[
     -   \liminf_{\eps\down 0}   (1{+}\eps^{(1-\rho)/2}) C\eps^{(1-\rho)/2} =0, 
   \]
  we infer the thesis and conclude the proof of the lemma.
   \end{proof}

  \begin{remark}
  \label{rk:rho}
  \slshape
The calculations in the proof of the Lemma \ref{l:liminf-crucial} reveal that it is necessary to distinguish between the cases $q <4$ and $q \geq 4$. In order to handle both cases, the exponent $\rho$
appearing in the rescaling factor in \eqref{Eeps-3} must be chosen to satisfy $0<\rho< 2-\tfrac4\lambda$. 
Note that, if $ q\geq 4$, then the constraint on $\rho$ is simply $\rho\in ]0,1[$. If $q<4$ (which may occur only if $d\leq 3$),  then the constraint is meaningful only if $q>2$, which is always true.
We point out, in particular, that condition $0<\rho< 2-\tfrac4\lambda$ implies $0<\rho<1$, a requirement that plays a crucial role in the proof of Claim~2 in Lemma \ref{l:liminf-crucial}.  
  \end{remark}
  
 We are now in a position to prove  \underline{Proposition \ref{prop:limit-jump}}.
  \begin{proof}
  For notational simplicity, hereafter we shall write $\eps$ in place of $\eps_n$. 
  
  First of all, we show that  if $  L_{\mathrm{bulk}}<\infty$, then $\res{z} \JUMP{\uu}=0$ a.e.\ in $\GC$. For this, 
  we bring into play \eqref{liminf-crucial} and \eqref{true-lsurf} and   from Lemma \ref{l:liminf-crucial} and
deduce  that, up to a subsequence  there holds
  \[
  C\geq \sup_{\eps>0} \int_{\GC}  \frac\mu{2\eps} 
  \EEE \  |\sqrt{\gr^\eps(\res{z}_\eps(1,x'))}\uu_\eps(\eps,x'){-}\sqrt{\gr^\eps(\res{z}_\eps(-1,x'))}\uu_\eps({-}\eps,x') |^2 \dd x'   \,.
  \]
  Taking into account  \eqref{toshow} from 
 Lemma \ref{l:pedante}, up to a further subsequence we deduce that
  \begin{equation}
  \label{up2asubs}
\exists\, C'>0 \ \forall\, \eps>0\, : \qquad   C'\eps\geq \int_{\GC}   \gr^\eps(\res{z}(x'))|\uu(\eps,x'){-}\uu({-}\eps,x') |^2 \dd x'  \geq  \int_{\GC}   \res{z}^2(x')|\uu(\eps,x'){-}\uu({-}\eps,x') |^2 \dd x' \,.
  \end{equation}
 
Now,   recall that  for almost all $x'\in \GC$
  \[
  \JUMP{\uu}(x') = \uu_+(x') -  \uu_-(x') \qquad \text{with } \uu_\pm \text{ the traces of $\uu|_{\Omega_\pm}$ on $\GC$}. 
  \]
  It follows from \eqref{trivial-but-useful} that 
  \[
  \uu({\pm} \eps,\cdot) \longrightarrow \uu_\pm \quad \text{in } \mathrm{L}^2(\GC) \text{ as } \eps \downarrow 0.
  \]
  Therefore, from \eqref{up2asubs} we gather
  \[
  0=  \lim_{\eps\down 0} \int_{\GC}   \res{z}^2(x')|\uu(\eps,x'){-}\uu({-}\eps,x') |^2  \dd x' \geq  \int_{\GC}   \res{z}^2(x') |\JUMP{\uu}(x')|^2 \dd x' \,.
  \]
  
  It remains to show that 
  \[
   L_{\mathrm{bulk}} \geq  \int_{\GC}    |\JUMP{\uu}(x')|^2 \dd x'\,.
  \]
  For this, we again  to resort to Lemma \ref{l:liminf-crucial}: we have
  \[
     L_{\mathrm{bulk}}  \geq \liminf_{\eps \down 0} \int_{\GC}  \frac{\mu}{2\eps} \EEE  \gr^\eps(\res{z}_\eps(1,x')) |\uu_\eps(\eps,x'){-}\uu_\eps({-}\eps,x') |^2 \dd x'   \geq  \liminf_{\eps \down 0} \int_{\GC}   \frac\mu2 \EEE  (1+\res{z}_{\eps}^2(1,x'))|\uu_\eps(\eps,x'){-}\uu_\eps({-}\eps,x') |^2 \dd x'  
  \]
  where we have used that  $\frac1\eps \gr^\eps(\res{z}_\eps) \geq \res{z}^2_{\eps}+ 1$ a.e.\ in $\OD$. 
Recalling \eqref{final-conclusion} and taking into account that $\res{z} \JUMP{u} =0$ on $\GC$,  we  end up with 
  \[
    L_{\mathrm{bulk}}   \geq  \liminf_{\eps \down 0} \int_{\GC} 
     \frac\mu2 \EEE  (1+\res{z}^2_{\eps})|\uu(\eps,x'){-}\uu({-}\eps,x') |^2 \dd x'   \geq \int_{\GC}   \frac\mu2 \EEE   |\JUMP{\uu}(x')|^2 \dd x' ,
  \]
  as desired.

  \end{proof}
  
We conclude this section by proving condition \eqref{gamma-en}.

\begin{proposition}
\label{prop:proof-C4}
 Assume $0<\rho< 2-\tfrac4\lambda$,
 with $\lambda = \min\{4, q\}$. \EEE
Let $(t_n,\uu_n,\res{z}_n)_n, \, (t,\uu,\res{z}) \in [0,T]{\times}\resX$ be
such that 
\[
t_n \to t,\  (\uu_n, \res{z}_n) \weaktoX (\uu,\res{z}) \text{ in } \resX, \  (\uu_n, \res{z}_n) \in \res{S}_{\eps_n}(t_n) \ \forall\, n \in \N.
\]
Then,
 \begin{equation}
 \label{gamma-en-proved}
 \res{E}(t,\uu,\res{z}) \leq \liminf_{n\to\infty} \resen(t_n,\uu_n,\res{z}_n).
 \end{equation} 
\end{proposition}
  \begin{proof}
  Recalling \eqref{resE-eps}, \eqref{eq:resQep}, \eqref{En-lim-1}, and \eqref{En-lim-2} from \eqref{Tau_u_cvg} we directly infer
  \begin{equation}
  \label{eq:lscElastic}
  \int_{\Omega_- {\cup} \Omega_+ }  \rmW(\nabla \uu) \dd x +\res{G}(\uu,\res{z})  - \langle \effe(t), \uu \rangle_{\mathrm{H}^1(\Omega;\R^d)} \leq \liminf_{n\to\infty} \res{Q}_{\mathrm{el}}^{\eps_n}(t,\uu_n,\res{z}_n), 
  \end{equation}
  whereas property \eqref{upphi}, \eqref{En-lim-3}, the rescaling in \eqref{damageable-rescaled}--\eqref{damageable-rescaled-1/2}, and Proposition \ref{prop:limit-jump} entail
  \begin{equation}
\label{eq:lscZeta}
\res{F} (\res{z})\leq \liminf_{n\to\infty} \res{F}^{\eps_n} (\res{z}_n).
  \end{equation}
 The statement follows then by combining \eqref{eq:lscElastic} with \eqref{eq:lscZeta}.
  \end{proof}
%
  \section{Closedness of the stable sets}
  \label{s:6}
%
Recall the definition of the functionals $\calQ_\mathrm{el}^\eps,$ $\calF^\eps,$ $\calE_\eps$ from \eqref{Eeps}, 
 of their rescaled counterparts $\res{Q}_\mathrm{el}^\eps,$ $\res{F}^\eps,$ $\res{E}_\eps$ from 
 \eqref{damageable-rescaled}--\eqref{resZ-eps},  and of   $\res{F},$ $\res{E}$ from \eqref{En-lim}. 
For the discussion in this section we will use  the  notation 
  \[
  \begin{cases}
  \calQ_{\mathrm{el}}^{\gr^\eps} & \text{in place of }   \calQ_\mathrm{el}^\eps,
  \\
  \res{Q}_\mathrm{el}^{\gr^\eps} & \text{in place of }   \res{Q}_\mathrm{el}^\eps,
  \\
  \res{E}_{\gr^\eps} & \text{in place of }  \res{E}_\eps
  \end{cases}
  \]
 to explicitly highlight the dependence of the energy and the results on the choice of the function $\gr^\eps$. 
 We will also need to introduce the following energy functionals
 \EEE
\begin{subequations}
\label{Ens-Sec6}
\begin{align}
    \res{H}_{\gr^\eps}:  [0,T]\times \rmH_{\GDir}^1(\Omega_{-}{\cup}\Omega_+){\times}  \rmW^{1,r}(\OD) \EEE \to[0,\infty]\,,\;\;
    \res{H}_{\gr^\eps}(t,\uu,\res z)&:=
    {\res{Q}}_\mathrm{el}^{\gr^\eps}(t,\uu,\res z)
    +\res{J}_\eps(\res z)\;\;\text{ with }
\\
    \res{J}_\eps: \rmW^{1,r}(\OD)\EEE \to\R\,,\;\;\qquad
    \res{J}_\eps(\res z)&:=\int_{\Omega_\mathrm{D}}
    \left(  \tfrac{1}{r\eps} \EEE |\partial_{{\res r}_1}\res z|^r{+}  \tfrac{1}{r} \EEE |\nabla_{x'}\res z|^r{+}\upphi(\res z) \right) \,\mathrm{d} x'\,;
\end{align}
\end{subequations}
 Indeed, $ \res{H}_{\gr^\eps}$ and $\res{E}_{\gr^\eps}$  only differ via the contributions
$ \res{J}_\eps$ and 
 $\res{F}_\eps = \res{F}_\eps(\res z) =  \tfrac{\eps^\rho}q   \res{A}_{q}^\eps(\res{z})  +  \tfrac12\res{A}^\eps(\res{z}) 
+ \int_{\OD}  
\upphi (\res z)  \dd \res{x}$.  
 As we will see, 
it is helpful to  bring into the picture  functional involving  a general Sobolev gradient of integrability $r\in(1,\infty)$ for better comparison with the results from \cite{MiRoTh10DDNE} where $\gr^\eps(z)\geq \eps^\gamma$ with $\gamma\neq 1$. 
\par
Indeed, we start by recalling a  
%
%
%
%
 result on the existence of mutual recovery sequences  that  can be taken from  \cite{MiRoTh10DDNE} to hold  for the above setting. Note that, in the statement below, we will take the exponent $\gamma \in (1,\infty)$. Thus, 
the corresponding limit energy, to be compared with $\mathfrak{E}$ from \eqref{En-lim},   will be 
\begin{subequations}
\label{En-Marita}
\begin{equation}
 \res{H}(t,\uu, \res{z}) : =   \int_{\Omega_- {\cup} \Omega_+ }  \rmW(\nabla \uu) \dd x +\res{K}(\uu,\res{z}) +\res{J} (\res{z}) - \langle \effe(t), \uu \rangle_{\mathrm{H}^1(\Omega;\R^d)} 
\end{equation}
with  $\res{J} (\res{z})  = \int_{\Omega_\mathrm{D}}
(\tfrac{1}{r} \EEE |\nabla_{x'}\res z|^r{+}\upphi(\res z) ) \,\mathrm{d} x'$ 
and  the coupling energy
    \begin{equation}
 \label{En-lim-2-Marita}
 \res{K} :  \rmH_{\GDir}^1(\Omega_{-}{\cup}\Omega_+){\times}  \mathrm{L}^1(\OD) \to [0,+\infty), \qquad \res{K}(\uu, \res{z}): = \int_{\GC} \left (I_{\{\boldsymbol{0}\}}(\res{z} \JUMP{\uu})\right) \dd x'\,.
 \end{equation}
 \end{subequations}
In fact, we remind the reader that  in the setup of   \cite{MiRoTh10DDNE} only the brittle constraint  coupled the evolution of  $\res z$ and $\uu$ on $\GC$. \EEE
 %

\begin{lemma}[Summary of {\cite[Sec.\ 3.2, Lemma 5\,+\,proof]{MiRoTh10DDNE}} for the setting of \eqref{Ens-Sec6}]
\label{Lemma-MRS-OLD}
Let $r\in(1,\infty)$ and $\gr^\eps(z)\geq \eps^\gamma$ for some constant $\gamma\in(1,\infty)$ and denote 
\begin{equation*}
\res{X}_r:=\rmH^1_{\Gamma_\mathrm{Dir}}(\Omega_-{\cup}\Omega_+;\R^d)\times \rmW^{1,r}(\Omega_\mathrm{D})\,.
\end{equation*}
Let $(t_\eps,\uu_\eps,\res{z}_\eps)$ be a stable sequence for the rate-independent systems $(\res{X}_r,\res{H}_{\gr^\eps},\res{R}_\eps)$ such that 
\begin{equation*}
t_\eps\to t\;\text{ in }[0,T]\,,\quad
\uu_\eps\overset{\tau_U}{\longrightarrow}\uu\,,\;
\text{and}\quad
\res{z}_\eps\rightharpoonup\res{z}\;\text{ in }\rmW^{1,r}(\Omega_D)\,.
\end{equation*}
Consider $(\hat \uu,\hat{\res{z}})\in\res{X}_r$ such that $\partial_{x_1}\hat{\res{z}}\equiv0$ a.e.\ in $\Omega_\mathrm{D}$. With $\hat\uu^\pm:=\hat\uu|_{\Omega_\pm},$ $I_\eps^+:=[0,\eps)$ as well as $I_\eps^-:=(-\eps,0]$ for all $\eps>0,$ introduce the reflections of $\hat\uu^\pm|_{I_\eps^\pm}$ along the interface $\{0\}\times\GC$ as well as their interpolation $\intp^\eps(\hat \uu)\in \rmH^1(\Omega^\eps_D,\R^d)$ by 
\begin{equation*}
\intp^\eps(\hat\uu):=
\frac{\eps-x_1}{2\eps}\hat \uu^-(\pm x_1,s)
+\frac{\eps+x_1}{2\eps}\hat \uu^+(\mp x_1,s)
\quad\text{for }x_1\in I_\eps^\mp
\end{equation*}
and define the function 
$\recop{\uu}{\eps}(\hat \uu)\in \rmH^1(\Omega,\R^d)$ as follows
\begin{subequations}
\label{mrs-old}
\begin{equation}
\label{mrs-old-u}
\recop{\uu}{\eps}(\hat \uu)
(x_1,s):=\left\{
\begin{array}{ll}
\hat\uu^\pm(x_1,s)&\text{ if }(x_1,s)\in\Omega^\eps_\pm\,,\\
\intp^\eps(\hat\uu)(x_1,s)&\text{ if }(x_1,s)\in\Omega^\eps_D\,.
\end{array}
\right.
\end{equation}
Furthermore, for $\hat{\res{z}}\in \rmW^{1,r}(\Omega_D)$ with  $\partial_{x_1}\hat{\res{z}}\equiv0$ a.e.\ in $\Omega_\mathrm{D}$ define the function $\recop{\res{z}}{\eps}(\hat{\res{z}})\in \rmW^{1,r}(\Omega_D)$ by 
\begin{equation}
\label{mrs-old-z}
\recop{\res{z}}{\eps}(\hat{\res{z}}):=\max\Big\{0,\min
\big\{\res{z}_\eps,\hat{\res{z}}-\delta_\eps\big\}\Big\}
\end{equation}
for $\delta_\eps=o(\|\res{z}_\eps-\res{z}\|_{\rmL^r(\Omega_D)})$ determined by Markov's inequality $\mathrm{(M)}$ such that 
\begin{equation*}
\calL^d\Big([|\res{z}_\eps=\res{z}|>\delta_\eps]\Big)
\overset{\text{\rm{(M)}}}{\leq}
\delta_\eps^{-r}\|\res{z}_\eps-\res{z}\|_{\rmL^r(\Omega_D)}^r\,\overset{!}{\to}0\,.
\end{equation*}
\end{subequations}
Then, for all   $(\hat \uu,\hat{\res{z}})\in\res{X}_r$ \EEE  such that $\partial_{x_1}\hat{\res{z}}\equiv0$ a.e.\ in $\Omega_\mathrm{D}$ the sequence 
$(\recop{\uu}{\eps}(\hat \uu),\recop{\res{z}}{\eps}(\hat{\res{z}}))_\eps$ constructed by \eqref{mrs-old} is a mutual recovery sequence, i.e.\ it satisfies
\begin{equation}
\label{MRSC-old}
\limsup_{\eps\to0}\Big(
\res{H}_{\gr^\eps}(t_\eps,\recop{\uu}{\eps}(\hat \uu),\recop{\res{z}}{\eps}(\hat{\res{z}}))
+\res{R}_\eps(\recop{\res{z}}{\eps}(\hat{\res{z}})-\res{z}_\eps)
-\res{H}_{\gr^\eps}(t_\eps,\uu_\eps,\res z_\eps)
\Big)
\leq
\Big(
\res{H}(t,\hat\uu,\hat{\res{z}})
+\res{R}(\hat{\res{z}}-\res{z})
-\res{H}(t,\uu,\res z)
\Big)\,,
\end{equation}
and in particular, along a not relabeled subsequence there holds as  $\eps\to0$ \EEE
\begin{equation}
\label{conv-mrs-old}
 \recop{\uu}{\eps}(\hat \uu)\overset{\tau_U}{\longrightarrow}\hat\uu 
 \text{ in } \rmH^1_{\Gamma_\mathrm{Dir}}(\Omega_-{\cup}\Omega_+;\R^d) 
\quad
\text{and}\quad
 \recop{\res{z}}{\eps}(\hat{\res{z}})\rightharpoonup\hat{\res{z}}\;\text{ in }\rmW^{1,r}(\Omega_D), \EEE
\end{equation}
as well as 
\begin{subequations}
\label{props-ens-old}
\begin{align}
\label{props-ens-old-grad}
\hspace*{-3ex}
\limsup_{\eps\to0}\Big(
\res{J}_\eps(\recop{\res{z}}{\eps}(\hat{\res{z}}))-\res{J}_\eps(\res{z}_\eps)
\Big)
\leq
\limsup_{\eps\to0}\res{J}_\eps(\recop{\res{z}}{\eps}(\hat{\res{z}}))
-\liminf_{\eps\to0}\res{J}_\eps(\res{z}_\eps)
&\leq
\res{J}(\hat{\res{z}})-\res{J}(\res z)\,,\\
\limsup_{\eps\to0}
\int_{\Omega_+^\eps\cup\Omega_-^\eps}\big(\rmW(e(\recop{\uu}{\eps}(\hat \uu)))-\rmW(e(\uu_\eps))\big)
\,\mathrm{d}x
\label{props-ens-old-W}
&\leq
\int_{\Omega_+\cup\Omega_-}\rmW(e(\hat \uu))-\rmW(e(\uu))
\,\mathrm{d}x
\,,\\
\hspace*{-5ex}
\limsup_{\eps\to0}
\int_{\Omega_\mathrm{D}^\eps}\!\!\!
\WD(\recop{\res{z}}{\eps}(\hat{\res{z}}){\circ}\mathrm{T}_\eps^{-1},e(\recop{\uu}{\eps}(\hat \uu)))
-\WD(\res{z}_\eps{\circ}\mathrm{T}_\eps^{-1},e(\uu_\eps))
\,\mathrm{d}x
\label{props-ens-old-WD}
&\leq
\int_{\GC}\!\!\!
\big(
I_{\{0\}}(\hat{\res{z}}\JUMP{\hat \uu})
-I_{\{0\}}({\res{z}}\JUMP{\uu})
\big)
\,\mathrm{d}x'
\,,\\
\res{R}_\eps(\recop{\res{z}}{\eps}(\hat{\res{z}})-\res{z}_\eps)&\longrightarrow\res{R}(\hat{\res{z}}-\res{z})\,.
\end{align}
\end{subequations}
\end{lemma}
\noindent 
The above result applies in particular for $r=2$ as well as for $r=q$ and it ensures that the constructed recovery sequence $(\res{z}_\eps)_\eps$ lies in $\rmW^{1,r}(\Omega_D)$. 
\par
Our task is now to adapt the above construction \eqref{mrs-old} to the setting $\gamma=1$ for the rate-independent systems $(\res{X}_\eps,\res{E}_{\gr^\eps},\res{R}_\eps)_\eps$ and its limit  $(\res{X},\res{E},\res{R})$.  Namely, for all $(\hat{\uu},\hat{\res{z}})\in\res{X}$   we will have to exhibit  \EEE  a mutual recovery sequence $(\recop{\uu}{\eps}(\hat \uu),\recop{\res{z}}{\eps}(\hat{\res{z}}))_\eps\subset{\res{X}}_\eps$ and  show in analogy to \eqref{MRSC-old} that the following mutual recovery sequence condition is satisfied: 
\begin{equation}
\label{MRSC-new}
\limsup_{\eps\to0}\Big(
{\res{E}}_{\gr^\eps}(t_\eps,\recop{\uu}{\eps}(\hat \uu),\recop{\res{z}}{\eps}(\hat{\res{z}}))
+\res{R}_\eps(\recop{\res{z}}{\eps}(\hat{\res{z}})-\res{z}_\eps)
-{\res{E}}_{\gr^\eps}(t_\eps,\uu_\eps,\res z_\eps)
\Big)
\leq
\Big(
{\res{E}}(t,\hat{\uu},\hat{\res{z}})
+\res{R}(\hat{\res{z}}-\res{z})
-{\res{E}}(t,\uu,\res z)
\Big)\,.
\end{equation}
Due to the presence of the $q$-Laplacian energy term in $\res{E}_\eps$ that scales with $\eps^\rho,$ 
this requires in particular to suitably mollify a competitor $\hat{\res{z}}\in\res{Z}$ before applying the construction \eqref{mrs-old-z}. This mollification will be done in such a controlled way that 
the regularization term still vanishes 
as $\eps\to0$ despite the blow up of the $\rmW^{1,q}$-gradient. Furthermore, in order to deduce \eqref{MRSC-new} it also has to be 
verified that relations alike \eqref{props-ens-old} are satisfied for the adapted mutual recovery sequence also in the setting $\gamma=1$.  
\par 
Accordingly, we gather the result on the mutual recovery sequence for $\gamma=1$ in the following proposition: 
\begin{proposition}
Let $\gr^\eps(z)=z^2+\eps^\gamma$ with $\gamma=1$. 
 Let $q>d$ and  $0<\rho< 2-\tfrac4\lambda$,
 with $\lambda = \min\{4, q\}$. \EEE 
Let $(t_\eps,\uu_\eps,\res{z}_\eps)$ be a stable sequence for the rate-independent systems $(\res{X}_\eps,\res{E}_{\gr^\eps},\res{R}_\eps)$ such that 
\begin{equation}
\label{conv-stab-seq-new}
t_\eps\to t\;\text{ in }[0,T]
\quad\text{and}\quad
(\uu_\eps,\res{z}_\eps)\overset{\res{X}}{\rightharpoonup} (\uu,\res{z})\,.
\end{equation}
Fix  $(\hat \uu,\hat{\res{z}})\in\res{X}$  such that $\hat{\res{z}}\in[0,1]$ a.e.\ in $\Omega_D$.  \EEE 
For $d\in\N$ denote by $B^d_1(0)\subset\R^d$ 
the open ball in $\R^d$ with center in $0\in\R^d$ and radius $1$.   Denote by 
\begin{equation*}
\eta_1(x'):=\left\{\begin{array}{ll}
k\exp(-1/(1-|x'|^2))&\text{ if }|x'|<1\,,\\
0&\text{ otw.\,}
\end{array}\right.
\end{equation*}
the standard mollifier with $x'\in\R^{d-1}$ and normalization constant $k>0$ and  
\begin{equation}
\label{def-molli}
\eta_{g(\eps)}(x'):=\frac{1}{g(\eps)^{d-1}}\,\eta_1(x'/g(\eps))\quad\text{ with }
 g(\eps):=\eps^{\rho/(2q)}\EEE 
\,.
\end{equation}
 We extend $\hat{\res{z}}$ from $\rmH^1(\Omega_D)$ to $\rmH^1(\Omega_D+B_1^d(0))$ such that also $\partial_{x_1}\hat{\res{z}}_1\equiv0$ and $\hat{\res{z}}_1\in[0,1]$ a.e.\ in $\Omega_D+B_1^d(0)$. We denote this extension by $\hat{\res{z}}_1$ and its trace to 
$\GC+B_1^{d-1}(0)$ by $\hat{\res{z}}_1|_{\GC+B_1^{d-1}(0)}$  and set 
\begin{equation}
\label{mrs-new-reg-z}
N_\eps(\hat{\res{z}})({\res{r}}_1,x'):=(\hat{\res{z}}_1|_{\GC+B_1^{d-1}(0)}*\eta_{g(\eps)})(x')\quad\text{for all }({\res{r}}_1,x')\in\Omega_D\,.
\end{equation}
\begin{subequations}
\label{mrs-new}
We define $\recop{\res{z}}{\eps}(\hat{\res{z}})$  (with slight abuse of notation) \EEE  by 
\begin{equation}
\label{mrs-new-z}
\recop{\res{z}}{\eps}(\hat{\res{z}}):=\recop{\res{z}}{\eps}(N_\eps(\hat{\res{z}}))
=\max\Big\{0,\min
\big\{\res{z}_\eps,N_\eps(\hat{\res{z}})-\delta_\eps\big\}\Big\}
\end{equation}
with $\delta_\eps$ determined as in \eqref{mrs-old-z} with $r=2$. Like in \eqref{mrs-old-u} we also define $\recop{\uu}{\eps}(\hat \uu)\in \rmH^1(\Omega,\R^d)$ by 
\begin{equation}
\label{mrs-new-u}
\recop{\uu}{\eps}(\hat \uu)
(x_1,s):=\left\{
\begin{array}{ll}
\hat\uu^\pm(x_1,s)&\text{ if }(x_1,s)\in\Omega^\eps_\pm\,,\\
\intp^\eps(\hat\uu)(x_1,s)&\text{ if }(x_1,s)\in\Omega^\eps_D\,.
\end{array}
\right.
\end{equation}
\end{subequations}
Then, for all $(\hat \uu,\hat{\res{z}})\in\res{X}$ such that $\partial_{x_1}\hat{\res{z}}\equiv0$ a.e.\ in $\Omega_\mathrm{D}$ there holds 
\begin{equation}
\label{space-mrs-new}
\recop{\uu}{\eps}(\hat \uu)\in \rmH^1_{\Gamma_D}(\Omega)
\quad\text{and}\quad
\recop{\res{z}}{\eps}(\hat{\res{z}})\in \rmW^{1,q}(\Omega_D) 
\end{equation}
and the sequence 
$(\recop{\uu}{\eps}(\hat \uu),\recop{\res{z}}{\eps}(\hat{\res{z}}))_\eps$ constructed in \eqref{mrs-new} is a mutual recovery sequence, i.e.\ it satisfies the mutual recovery sequence condition \eqref{MRSC-new}. In particular, along a not relabeled subsequence there holds 
\begin{equation}
\label{conv-mrs-new}
 (\recop{\uu}{\eps}(\hat \uu),\recop{\res{z}}{\eps}(\hat{\res{z}}))\overset{\res{X}}{\rightharpoonup} (\hat\uu,\hat{\res{z}})\
\end{equation}
and the following statements hold true:
\begin{subequations}
\label{props-ens-new}
\begin{align}
\label{props-ens-new-1}
\limsup_{\eps\to0}\Big(
 \res{F}_\eps \EEE (\recop{\res{z}}{\eps}(\hat{\res{z}}))- \res{F}_\eps \EEE (\res{z}_\eps)
\Big)
&\leq
 \res{F}(\hat{\res{z}})-\res{F}(\res z) \EEE \,,\\
\nonumber
\limsup_{\eps\to0}\Big(
\int_{\Omega_+^\eps\cup\Omega_-^\eps}\rmW(\nabla\recop{\uu}{\eps}(\hat \uu))&-\rmW(\nabla\uu_\eps)
\,\mathrm{d}x-\langle\BS{f}(t),\recop{\uu}{\eps}(\hat \uu)-\uu_\eps\rangle\Big)
\\
\label{props-ens-new-2a}
&\leq
\int_{\Omega_+\cup\Omega_-}\rmW(\nabla\hat \uu)-\rmW(\nabla\uu)
\,\mathrm{d}x-\langle \BS{f}(t),\hat \uu-\uu\rangle
\,,\\
\nonumber
\limsup_{\eps\to0}\Big(
\int_{\Omega_\mathrm{D}^\eps}
\WD(\recop{\res{z}}{\eps}(\hat{\res{z}}){\circ}\mathrm{T}_\eps^{-1}&,\nabla\recop{\uu}{\eps}(\hat \uu))
-\WD(\res{z}_\eps{\circ}\mathrm{T}_\eps^{-1},\nabla\uu_\eps)
\,\mathrm{d}x\Big)\\
\label{props-ens-new-2b}
&\leq
\int_{\GC}
\big(
I_{\{0\}}(\hat{\res{z}}\JUMP{\hat \uu})+\frac{\mu\EEE}{2}\big|\JUMP{\hat \uu}\big|^2
-I_{\{0\}}({\res{z}}\JUMP{\uu})-\frac{\mu\EEE}{2}\big|\JUMP{\uu}\big|^2
\big)
\,\mathrm{d}x'
\,,\\
\label{props-ens-new-3}
\res{R}_\eps(\recop{\res{z}}{\eps}(\hat{\res{z}})-\res{z}_\eps)&\longrightarrow\res{R}(\hat{\res{z}}-\res{z})\,,\\
\label{props-ens-new-4}
\limsup_{\eps\to0}\eps^\rho\Big(
\res{A}_q^\eps(\recop{\res{z}}{\eps}(\hat{\res{z}}))
-\res{A}_q^\eps(\res{z}_\eps)
\Big)
&=0\,.
\end{align}
\end{subequations}
\end{proposition}
\begin{proof}
Fix   $(\hat \uu,\hat{\res{z}})\in\res{X}$ with  $\hat{\res{z}} \in [0,1]$ \EEE  a.e.\ in $\Omega_\mathrm{D}$. 
\par
{\bf Proof of \eqref{space-mrs-new}: }
For $\hat{\res{z}}\in \rmH^1(\Omega_D)$   as above \EEE 
the mollification \eqref{mrs-new-reg-z} results in the function 
$N_\eps(\hat{\res{z}})\in C^\infty(\Omega_D)$ with $\partial_{x_1}\hat{\res{z}}\equiv0$. 
In particular, this function also satisfies $N_\eps(\hat{\res{z}})\in \rmW^{1,q}(\Omega_D)$ and then, by construction \eqref{mrs-new-z} also 
\begin{equation}
\label{r1-der-Neps}
\partial_{{\res{r}}_1}\recop{\res{z}}{\eps}(\hat{\res{z}})\equiv0
\quad\text{in }\Omega_D \qquad \text{and}   \qquad \recop{\res{z}}{\eps}(\hat{\res{z}}) \in [0,1] \text{ a.e.\ in $\Omega_\mathrm{D}$.} \EEE
\end{equation}
The result \cite{Marcus-Mizel} then ensures that  $\recop{\res{z}}{\eps}(\hat{\res{z}})\in \rmW^{1,q}(\Omega_D)$ for all $\eps>0$ with 
\begin{equation}
\label{grad-mrs-z}
\nabla \recop{\res{z}}{\eps}(\hat{\res{z}})(\res{x})=\left\{
\begin{array}{cl}
\nabla N_\eps(\hat{\res{z}})(\res{x})
&\text{ if }\res{x}\in A_\eps:=[0\leq N_\eps(\hat{\res{z}})-\delta_\eps\leq \res{z}_\eps]\,,\\
\nabla\res{z}_\eps(\res{x})&\text{ if }\res{x}\in B_\eps:=[\res{z}_\eps<N_\eps(\hat{\res{z}})-\delta_\eps]\,,\\
0&\text{ if }\res{x}\in C_\eps:=\Omega_D\backslash(A_\eps\cup B_\eps)\,.
\end{array}
\right.
\end{equation}
\par
{\bf Proof of convergence results \eqref{conv-mrs-new}: }
The convergence result for $(\recop{\uu}{\eps}(\hat \uu))_\eps$ can be directly taken from \eqref{conv-mrs-old}. 
To verify the convergence result for 
$(\recop{\res{z}}{\eps}(\hat{\res{z}}))_\eps$ we argue as follows. 
For $g(\eps)$ in \eqref{def-molli} we observe that 
$g(\eps)\to0$ as $\eps\to0$. Moreover, $\overline{\Omega_D}\Subset\Omega_D+B_1(0)$. Hence, by the properties of the mollification by convolution of Sobolev functions, cf.\ e.g.\ \cite[Lemma 3.15, p.\ 52]{Adams}, there holds 
$(\hat{\res{z}}_1*\eta_{g(\eps)})\to\hat{\res{z}}_1$ in $\rmH^1(\Omega_D),$ where  $\hat{\res{z}}_1\equiv\hat{\res{z}}$ in $\rmH^1(\Omega_D)$. Accordingly, we find that 
\begin{equation}
\label{conv-Neps-H1}
N_\eps(\hat{\res{z}})\to \hat{\res{z}}\quad\text{in }\rmH^1(\Omega_D)\, 
\end{equation}
and together with convergence result \eqref{conv-mrs-old} we also conclude the existence of a not relabeled subsequence along which 
\begin{equation}
\label{conv-mrs-new-z}
\recop{\res{z}}{\eps}(\hat{\res{z}})
:=\recop{\res{z}}{\eps}(N_\eps(\hat{\res{z}}))\rightharpoonup\hat{\res{z}}\quad\text{in }\rmH^1(\Omega_D)\,.
\end{equation}
For the latter result, one may follow the arguments 
of \cite[Lemma 5, Step 1]{MiRoTh10DDNE}, resp.\ \cite[Thm.\ 3.14, Step 1]{ThoMie09DNEM}, using that, by construction, the sequence $(\recop{\res{z}}{\eps}(\hat{\res{z}}))_\eps$ is uniformly bounded in $\rmH^1(\Omega_D),$ thus there exists a subsequence weakly converging to a limit $\tilde{\res{z}}$. Then, by the compact embedding of $\rmH^1(\Omega_D)$ in $\rmL^2(\Omega_D)$ and Riesz' convergence theorem, there follows the existence of a further subsequence converging pointwise a.e.\ to $\hat{\res{z}},$ which shows that $\tilde{\res{z}}=\hat{\res{z}}$. 
This finishes the proof of \eqref{conv-mrs-new}. 
\par
With these observations at hand we now first discuss convergence property 
\eqref{props-ens-new-3}, subsequently followed by properties \eqref{props-ens-new-4}, \eqref{props-ens-new-1}, \eqref{props-ens-new-2a}, and \eqref{props-ens-new-2b}. 
\par
{\bf Proof of \eqref{props-ens-new-3}: }Indeed, \eqref{props-ens-new-3} directly follows from convergence results \eqref{conv-mrs-new-z} and \eqref{conv-stab-seq-new}, thanks to the compact embedding of $\rmH^1(\Omega_D)$ in $\rmL^2(\Omega_D)$. 
\par 
{\bf Preparations for the proof of \eqref{props-ens-new-1} and \eqref{props-ens-new-4}: }
In what follows we carry out some estimates for the terms $\res{A}^\eps$ and $\res{A}^\eps_q$. To cover both exponents $2$ and $q,$ we use the placeholder  
$r\in\{2,q\}$ and write $\res{A}^\eps_r$. For shorter notation we also set 
 $\mathcal{A}^\eps_r(v):= |\frac{1}{\eps}\partial_{\res{r}_1}v|^r
+ |\nabla_{x'}v|^r$ \EEE for $v\in \rmW^{1,r}(\Omega_D)$. 
Following the arguments of \cite[Lemma 5, Step 1]{MiRoTh10DDNE}, 
resp.\ \cite[Thm.\ 3.14, Step 2]{ThoMie09DNEM}, we deduce from \eqref{grad-mrs-z} and \eqref{r1-der-Neps} that 
\begin{equation}
\label{grad-est-prep}
\begin{split}
\res{A}^\eps_r(\recop{\res{z}}{\eps}(\hat{\res{z}}))-\res{A}^\eps_r(\res{z}_\eps)
&=\int_{A_\eps}\mathcal{A}^\eps_r(N_\eps(\hat{\res{z}}))
-\mathcal{A}^\eps_r(\res{z}_\eps)\,\mathrm{d}\res{r_1}\,\mathrm{d}x'
-\int_{C_\eps}\mathcal{A}^\eps_r(\res{z}_\eps)\,\mathrm{d}\res{r_1}\,\mathrm{d}x'\\
&\leq \|\nabla_{x'}N_\eps(\hat{\res{z}})\|_{\rmL^r(A_\eps)}^r
-\|\nabla\res{z}_\eps\|_{\rmL^r(A_\eps\cup C_\eps)}^r\\
&\leq\|\nabla_{x'}N_\eps(\hat{\res{z}})\|_{\rmL^r(\Omega_D)}^r
-\|\nabla\res{z}_\eps\|_{\rmL^r(A_\eps\cup C_\eps)}^r
\end{split}
\end{equation}
Based on this estimate we now discuss the two cases $2$ and $q$ for the proofs 
of \eqref{props-ens-new-1} and \eqref{props-ens-new-4} separately. 
\par 
\par
{\bf Proof of \eqref{props-ens-new-1}: }
From \eqref{grad-est-prep} for $r=2$ we deduce that 
\begin{equation}
\label{props-ens-new-1-proof}
\begin{split}
\limsup_{\eps\to0}\Big(
\res{G}_\eps(\recop{\res{z}}{\eps}(\hat{\res{z}}))
-\res{G}_\eps(\res{z}_\eps)
\Big)
&=
\limsup_{\eps\to0}\Big(
\tfrac{1}{2}\res{A}^\eps_2(\recop{\res{z}}{\eps}(\hat{\res{z}}))
-\tfrac{1}{2}\res{A}^\eps_2(\res{z}_\eps)
+\int_{\Omega_D}
\big(\upphi(\recop{\res{z}}{\eps}(\hat{\res{z}}))
-\upphi(\res{z}_\eps)\big)\,\mathrm{d}\res{r}\Big)\\
&\leq 
\limsup_{\eps\to0}
\tfrac{1}{2}\|\nabla_{x'}N_\eps(\hat{\res{z}})\|_{\rmL^2(\Omega_D)}^2
-\liminf_{\eps\to0}\tfrac{1}{2}\|\nabla\res{z}_\eps\|_{\rmL^2(A_\eps\cup C_\eps)}^2\\
&\qquad
+\limsup_{\eps\to0}\int_{\Omega_D}
\big(\upphi(\recop{\res{z}}{\eps}(\hat{\res{z}}))
-\upphi(\res{z}_\eps)\big)\,\mathrm{d}\res{r}
\\
&\leq\tfrac{1}{2}\|\nabla_{x}\hat{\res{z}}\|_{\rmL^2(\Omega_D)}^2
-\tfrac{1}{2}\|\nabla\res{z}\|_{\rmL^2(\Omega_D)}^2
+\int_{\Omega_D}\big(\upphi(\hat{\res{z}})
-\upphi(\res{z})\big)\,\mathrm{d}\res{r}\\
&=\res{G}(\hat{\res{z}})-\res{G}(\res{z})
\,.
\end{split}
\end{equation}
The second estimate in \eqref{props-ens-new-1-proof} stems from the result that 
$-\liminf_{\eps\to0}\tfrac{1}{2}\|\nabla\res{z}_\eps\|_{\rmL^2(A_\eps\cup C_\eps)}^2
\leq -\tfrac{1}{2}\|\nabla\res{z}\|_{\rmL^2(\Omega_D)}^2,$ which can be directly taken from \eqref{props-ens-old-grad}, cf.\ \cite[Lemma 5, Step 1]{MiRoTh10DDNE}, 
resp.\ \cite[Thm.\ 3.14, Step 2]{ThoMie09DNEM}. Additionally, we used convergence result \eqref{conv-Neps-H1} to pass to the limit 
$\tfrac{1}{2}\|\nabla_{x'}N_\eps(\hat{\res{z}})\|_{\rmL^2(\Omega_D)}^2
\to\tfrac{1}{2}\|\nabla\hat{\res{z}}\|_{\rmL^2(\Omega_D)}^2$. Moreover we applied  convergence results \eqref{conv-mrs-new-z} and \eqref{conv-stab-seq-new} together with the compact embedding of $\rmH^1(\Omega_D)\cap \rmL^\infty(\Omega_D)$ in $\rmL^s(\Omega_D)$ for all $s\in(1,\infty)$  and the properties of $\upphi$ \EEE to conclude that 
\begin{equation*}
\limsup_{\eps\to0}\int_{\Omega_D}
\big(\upphi(\recop{\res{z}}{\eps}(\hat{\res{z}}))
-\upphi(\res{z}_\eps)\big)\,\mathrm{d}\res{r}
=\int_{\Omega_D}\big(\upphi(\hat{\res{z}})
-\upphi(\res{z})\big)\,\mathrm{d}\res{r}\,.
\end{equation*}
This finishes the proof of \eqref{props-ens-new-1}.
\par
{\bf Proof of \eqref{props-ens-new-4}: } 
In order to verify \eqref{props-ens-new-4} we estimate the last term 
in \eqref{grad-est-prep} for $r=q$ from above by $0$.  We claim that  
from \eqref{grad-est-prep} and thanks to the 
specific choice of the function $g(\eps)$ from \eqref{def-molli}, it holds \EEE that 
\begin{equation}
\label{props-ens-new-4-proof}
\frac{\eps^\rho}{q}\Big(
\res{A}^\eps_r(\recop{\res{z}}{\eps}(\hat{\res{z}}))-\res{A}^\eps_r(\res{z}_\eps)
\Big)
\leq \frac{\eps^\rho}{q}\|\nabla_{x'}N_\eps(\hat{\res{z}})\|_{\rmL^q(\Omega_D)}^q
\overset{(!)}{\longrightarrow}0\,.
\end{equation}
 In fact, in \EEE view of \eqref{def-molli} we observe for $i\in\{1,\ldots,d-1\}$ and $\tilde x_i'=x_i'/g(\eps)$ that 
\begin{equation}
\label{deriv-molli}
\partial_{x_i'} \eta_{g(\eps)}(x')
=\frac{1}{g(\eps)^{d-1}}\partial_{\tilde x_i'}\eta_1(\tilde x')
\partial_{x_i'}\tilde x_i'
=\frac{1}{g(\eps)^{d}}\partial_{\tilde x_i'}\eta_1(\tilde x')
\end{equation}
With this we further estimate the term in \eqref{props-ens-new-4-proof} as follows
\begin{equation*}
\begin{split}
\frac{\eps^\rho}{q}\|\nabla_{x'}N_\eps(\hat{\res{z}})\|_{\rmL^q(\Omega_D)}^q
&=\frac{\eps^\rho}{q}\int_{\Omega_D}\Big|\int_{\R^{d-1}}
\hat{\res{z}}_1(y')\nabla_{x'}\eta_{g(\eps)}(x'-y')
\,\mathrm{d}y'\Big|^q\,\mathrm{d}\res{r}_1\mathrm{d}x' \\
&\leq \|\hat{\res{z}}\|_{\rmL^\infty(\Omega_D)}^q
\|\nabla_{\tilde x'}\eta_1\|_{\mathrm{C}(\R^{d-1})}^q 
\frac{\eps^\rho}{q}\int_{\Omega_D}\Big|\int_{B_{g(\eps)}^{d-1}(x')}
\frac{1}{g(\eps)^d}\,\mathrm{d}y'\Big|^q\,\mathrm{d}\res{r}_1\,\mathrm{d}x'
\\
&=\|\nabla_{\tilde x'}\eta_1\|_{\mathrm{C}(\R^{d-1})}^q 
\calL^d(\Omega_D)
\frac{\eps^\rho}{q g(\eps)^q}
=\|\nabla_{\tilde x'}\eta_1\|_{\mathrm{C}(\R^{d-1})}^q 
\calL^d(\Omega_D)
\frac{\eps^{\rho/2}}{q}\;\to0\,,
\end{split}
\end{equation*}
where we used $g(\eps):=\eps^{\rho/(2q)}$ from \eqref{def-molli} and  $\rho>0$. \EEE This shows that convergence (!) in \eqref{props-ens-new-4-proof} indeed holds true, so that the proof of \eqref{props-ens-new-4} is finished.
\par
{\bf Proof of \eqref{props-ens-new-2a}: }Here we note that the construction of the recovery sequence $(\recop{\uu}{\eps}(\hat \uu))_\eps$ from \eqref{mrs-new-u} is identical to the one in \eqref{mrs-old-u}. Since also the stored elastic energy density $\rmW$ here has the same properties as in Lemma \ref{Lemma-MRS-OLD} we directly infer that estimate \eqref{props-ens-old-W} holds true also here.   
Moreover, thanks to property \eqref{force} there holds 
$\mathrm{supp} (\effe(t))\cap\Omega_\mathrm{D}^\eps=\emptyset$ in the sense of distributions for all $\eps>0$ and all $t\in[0,T]$. 
Hence, by construction \eqref{mrs-new-u} and convergence property \eqref{conv-mrs-new} we also have that 
\begin{equation}
\limsup_{\eps\to0}\langle \effe(t),\recop{\uu}{\eps}(\hat \uu)-\uu_\eps\rangle
=\langle \effe(t),\hat \uu\rangle-\liminf_{\eps\to0}\langle \effe(t),\uu_\eps\rangle
=\langle \effe(t),\hat\uu-\uu\rangle\,.
\end{equation}
Thus, together with \eqref{props-ens-old-W} we conclude \eqref{props-ens-new-2a}. 
\par
{\bf Proof of \eqref{props-ens-new-2b}: }Assume that $(\hat\uu,\hat{\res{z}})\in\res{X}$ satisfies the brittle constraint, i.e.
\begin{equation}
\label{props-ens-new-2b-brittle}
\hat{\res{z}}\big|_{\GC}\JUMP{\hat\uu}=0\quad\text{a.e.\ on }\GC\,.
\end{equation}
Otherwise $I_{0}(\hat{\res{z}}\JUMP{\hat\uu})=\infty$ a.e.\ on $\GC$ and \eqref{props-ens-new-2b} is trivially satisfied.
\par
In order to verify \eqref{props-ens-new-2b} given \eqref{props-ens-new-2b-brittle}, we follow the lines of the proof in \cite[Lemma 5, Step 3]{MiRoTh10DDNE} for property \eqref{props-ens-old-WD} and check that  
\eqref{props-ens-new-2b} is indeed the outcome in the case of $\gr^\eps(z):=z^2+\eps^\gamma$ with $\gamma=1$.  
We use that 
\begin{equation*}
\begin{split}
&\limsup_{\eps\to0}\Big(
\int_{\Omega_\mathrm{D}^\eps}
\WD(\recop{\res{z}}{\eps}(\hat{\res{z}}){\circ}\mathrm{T}_\eps^{-1},\nabla\recop{\uu}{\eps}(\hat \uu))
-\WD(\res{z}_\eps{\circ}\mathrm{T}_\eps^{-1},\nabla\uu_\eps)
\,\mathrm{d}x\Big)\\
&\leq \limsup_{\eps\to0}
\int_{\Omega_\mathrm{D}^\eps}
\WD(\recop{\res{z}}{\eps}(\hat{\res{z}}){\circ}\mathrm{T}_\eps^{-1},\nabla\recop{\uu}{\eps}(\hat \uu))\,\mathrm{d}x
-\liminf_{\eps\to0}\int_{\Omega_\mathrm{D}^\eps}
\WD(\res{z}_\eps{\circ}\mathrm{T}_\eps^{-1},\nabla\uu_\eps)
\,\mathrm{d}x\\
&\leq \limsup_{\eps\to0}
\int_{\Omega_\mathrm{D}^\eps}
\WD(\recop{\res{z}}{\eps}(\hat{\res{z}}){\circ}\mathrm{T}_\eps^{-1},\nabla\recop{\uu}{\eps}(\hat \uu)))\,\mathrm{d}x
-\int_{\GC}
\big(I_{\{0\}}({\res{z}}\JUMP{\uu})+\frac{\mu\EEE}{2}\big|\JUMP{\uu}\big|^2
\big)
\,\mathrm{d}x'
\,,
\end{split}
\end{equation*}
where the second estimate follows from the $\Gamma$-$\liminf$ estimate deduced in Proposition \ref{prop:limit-jump}. Hence, for \eqref{props-ens-new-2b} to hold it remains to show that 
\begin{equation}
\label{props-ens-new-2b-proof1}
\limsup_{\eps\to0}
\int_{\Omega_\mathrm{D}^\eps}
\WD(\recop{\res{z}}{\eps}(\hat{\res{z}}){\circ}\mathrm{T}_\eps^{-1},\nabla\recop{\uu}{\eps}(\hat \uu))\,\mathrm{d}x
\leq \int_{\GC}
\big(
I_{\{0\}}(\hat{\res{z}}\JUMP{\hat \uu})+\frac{\mu\EEE}{2}\big|\JUMP{\hat \uu}\big|^2
\big)
\,\mathrm{d}x'\,. 
\end{equation}
From the definition of the rescaling mapping $\mathrm{T}_\eps$ given in 
\eqref{def-resc} and the 
construction \eqref{mrs-new-z} it follows that $\gr^\eps(\recop{\res{z}}{\eps}(\hat{\res{z}})\circ\mathrm{T}_\eps^{-1})=\eps$ for all $\eps>0$ if $\hat{\res{z}}=0$. In view of the 
decomposition $\OD=A_\eps\cup B_\eps\cup C_\eps$ provided in \eqref{grad-mrs-z} for $\eps>0$ we can therefore decompose $\OD^\eps$ as 
$\OD^\eps=\mathrm{T}_\eps(A_\eps)\cup\mathrm{T}_\eps(B_\eps)
\cup\mathrm{T}_\eps(C_\eps)$. Thus, 
\begin{equation}
\label{Jump-1}
\begin{split}
&\int_{\OD^\eps}\gr^\eps(\recop{\res{z}}{\eps}(\hat{\res{z}}){\circ}\mathrm{T}_\eps^{-1})|\nabla\recop{\uu}{\eps}(\hat \uu)|^2\,\mathrm{d}x\\
&\leq 
\int_{\mathrm{T}_\eps(A_\eps)}\frac{\mu}{2}
\gr^\eps(N_\eps(\hat{\res{z}}){\circ}\mathrm{T}_\eps^{-1}{-}\delta_\eps)|\nabla\recop{\uu}{\eps}(\hat \uu)|^2\,\mathrm{d}x
+\int_{\mathrm{T}_\eps(B_\eps)}\frac{\mu}{2}(1+\eps)|\nabla\recop{\uu}{\eps}(\hat \uu)|^2\,\mathrm{d}x
+\int_{\mathrm{T}_\eps(C_\eps)}\frac{\mu}{2}\eps|\nabla\recop{\uu}{\eps}(\hat \uu)|^2\,\mathrm{d}x\,,
\end{split}
\end{equation}
where we used on $B_\eps$ that $0\leq \res{z}_\eps\leq1$ a.e.. 
\par
For a function $f\in \rmH^1(\OD)\cap \rmL^\infty(\OD)$ with  
let $L_f^c:=[f|_{\GC}=c]$ denote the level set of its trace $f|_{\GC}$ on $\GC$ for the value $c\in\R$.  
By the properties of the convolution there holds 
$L_{N_\eps(\hat{\res{z}})}^0\subset L_{\hat{\res{z}}}^0$ 
for all $\eps>0$ and we find for 
$\res{x}\in B_\eps=[N_\eps(\hat{\res{z}})-\delta_\eps>\res{z}_\eps]$ 
that $\gr^\eps(N_\eps(\hat{\res{z}})(\res{x})-\delta_\eps)>\eps$ and that  $B_\eps\cap\GC\subset\GC\backslash L_{\hat{\res{z}}}^0$. Similarly, we see that 
\begin{equation*}
A_\eps=[0\leq N_\eps(\hat{\res{z}})-\delta_\eps\leq \res{z}_\eps]
=[\delta_\eps\leq N_\eps(\hat{\res{z}})\leq \res{z}_\eps+\delta_\eps]
\subset[\delta_\eps\geq N_\eps(\hat{\res{z}})]\subset[0<N_\eps(\hat{\res{z}})]\,,
\end{equation*}
hence also $A_\eps\cap\GC\subset\GC\backslash L_{\hat{\res{z}}}^0,$ and $\gr^\eps(N_\eps(\hat{\res{z}})(\res{x}))-\delta_\eps)>\eps$ for $\res{x}\in A_\eps$. For 
$C_\eps=\OD\backslash(A_\eps\cap B_\eps)=[N_\eps(\hat{\res{z}})-\delta_\eps<0]$ have $\gr^\eps(N_\eps(\hat{\res{z}})(\res{x})-\delta_\eps)=\eps$ as well as  
$L_{\hat{\res{z}}}^0\subset C_\eps\cap\GC$. Thanks to these observations we can further estimate the terms on the right-hand side of \eqref{Jump-1}, thus arriving at 
\begin{equation}
\label{Jump-2} 
\begin{split}
&\int_{\OD^\eps}
\gr^\eps(\recop{\res{z}}{\eps}(\hat{\res{z}}){\circ}\mathrm{T}_\eps^{-1})
|\nabla\recop{\uu}{\eps}(\hat \uu)|^2
\,\mathrm{d}x\\
&\leq 
\int_{L_{\hat{\res{z}}}^0}\int_{-\eps}^\eps
\eps\,\frac{\mu}{2}|\nabla\recop{\uu}{\eps}(\hat \uu)|^2
\,\mathrm{d}x_1\,\mathrm{d}x'
+\int_{\GC\backslash L_{\hat{\res{z}}}^0}\int_{-\eps}^\eps
(1+\eps)\,\frac{\mu}{2}|\nabla\recop{\uu}{\eps}(\hat \uu)|^2
\,\mathrm{d}x_1\,\mathrm{d}x'\,, 
\end{split}
\end{equation}
where $|\nabla\recop{\uu}{\eps}(\hat \uu)|^2
\leq 2(|\partial_{x_1}\recop{\uu}{\eps}(\hat \uu)|^2
+|\nabla_{x'}\recop{\uu}{\eps}(\hat \uu)|^2)$.
\par
\begin{subequations}
\label{Jump-3}
With similar arguments as in \cite[Lemma 5, Step 3]{MiRoTh10DDNE} we show now that the second term 
in \eqref{Jump-2} satisfies 
\begin{equation}
\label{Jump-3-1}
\limsup_{\eps\to0}
\int_{\GC\backslash L_{\hat{\res{z}}}^0}\int_{-\eps}^\eps
(1+\eps)\,\frac{\mu}{2}|\nabla\recop{\uu}{\eps}(\hat \uu)|^2
\,\mathrm{d}x_1\,\mathrm{d}x'=0\,,
\end{equation}
whereas the first term in \eqref{Jump-2} results in 
\begin{equation}
\label{Jump-3-2}
\limsup_{\eps\to0}
\int_{L_{\hat{\res{z}}}^0}\int_{-\eps}^\eps
\eps\,\frac{\mu}{2}|\nabla\recop{\uu}{\eps}(\hat \uu)|^2
\,\mathrm{d}x_1\,\mathrm{d}x'
 \leq \EEE
\int_{L_{\hat{\res{z}}}^0}
\frac{\mu}{2}\big|\JUMP{\hat\uu}\big|^2
\,\mathrm{d}x'\,.
\end{equation}
\end{subequations}
Then, applying both $\limsup$-results \eqref{Jump-3} in \eqref{Jump-2} results in the claim \eqref{props-ens-new-2b-proof1}. 
\par
{\bf Proof of \eqref{Jump-3-1}: }Indeed, for \eqref{Jump-3-1} we can one by one repeat the arguments of \cite[Lemma 5, Step 3, (3.39)-(3.42)]{MiRoTh10DDNE}, which we do here for the sake of completeness. Doing so, it turns out that the change from $\gamma>1$ to $\gamma=1$ 
has no effect on the outcome in this limit passage $\eps\to0$, since in the factor $(1+\eps)$ on the left-hand side of \eqref{Jump-3-1} the constant $1$ dominates, which was also the  case for $\gamma>1$ discussed in \cite[Lemma 5, Step 3, (3.39)-(3.42)]{MiRoTh10DDNE}.   
\par
For notational simplicity let $\hat\uu^\pm$ denote also their even extensions to $\Omega$ by reflection at $x_1=0$. Then $\hat\uu^\pm\in \rmH^1(\Omega,\R^d)$. In view of  
$0<(\eps\pm x_1)/(2\eps)<1$ on $I_\eps^-\cup I_\eps^+$ construction \eqref{mrs-new-u} gives 
\begin{equation}
\label{Jump-3-1-1}
\|\nabla_{x'}\recop{\uu}{\eps}(\hat \uu)\|_{\rmL^2(\OD^\eps,\R^{(d-1)\times(d-1)})}
\leq2\|\nabla\hat\uu^-\|_{\rmL^2(I_\eps^-\times\GC,\R^{(d-1)\times(d-1)})}
+2\|\nabla\hat\uu^+\|_{\rmL^2(I_\eps^+\times\GC,\R^{(d-1)\times(d-1)})}\to0\,.
\end{equation}
Moreover, $\partial_{x_1}\recop{\uu}{\eps}(\hat \uu)=G_1^\eps+G_2^\eps$ with 
\begin{equation}
G_1^\eps
=\frac{\eps-x_1}{2\eps}\partial_{x_1}\hat\uu^-
+\frac{\eps+x_1}{2\eps}\partial_{x_1}\hat\uu^+
\quad\text{and}\quad
G_2^\eps=\frac{1}{2\eps}(\hat\uu^+-\hat\uu^-)\,.
\end{equation}
Like in \eqref{Jump-3-1-1} we also conclude that 
\begin{equation}
\label{Jump-3-1-2}
\|G_1^\eps\|_{\rmL^2(\OD^\eps,\R^{(d-1)\times(d-1)})}\to0
\end{equation}
as $\eps\to0,$ whereas $G_2^\eps$ needs further consideration. 
\par
Since $\hat\uu\in \rmH^1(\Omega\backslash L_{\hat{\res{z}}}^0,\R^d)$ there holds for a.e.\ $x'\in\GC\backslash L_{\hat{\res{z}}}^0$ that $\hat\uu^+(0,x')=\hat\uu^-(0,x')$ and hence, we obtain with the aid of H\"older's inequality
\begin{equation}
\label{Jump-3-1-3}
\begin{split}
\big|
\hat\uu^+(x_1,x')-\hat\uu^-(x_1,x')
\big|
&\leq 
\Big|\int_0^{x_1}\partial_\xi\hat\uu^+(\xi,x')\,\mathrm{d}\xi\Big|
+\Big|\int_0^{x_1}\partial_\xi\hat\uu^-(\xi,x')\,\mathrm{d}\xi\Big|
\\
&\leq |x_1|^{1/2}\Big(
\|\partial_{x_1}\hat\uu^+(\cdot,x')\|_{\rmL^2(I_\eps^+,\R^d)}
+\|\partial_{x_1}\hat\uu^-(\cdot,x')\|_{\rmL^2(I_\eps^-,\R^d)}
\Big)
\\
&=:A(\partial_{x_1}\hat\uu^+,\partial_{x_1}\hat\uu^+)
\,.
\end{split}
\end{equation}
Dividing by $2\eps$ and integrating over 
$(x_1,x')\in (I_\eps^-\cup I_\eps^+)\times\GC\backslash L_{\hat{\res{z}}}^0$ yields  
\begin{equation}
\label{Jump-3-1-4}
\begin{split}
\|G_2^\eps\|_{\rmL^2((I_\eps^-\cup I_\eps^+)\times(\GC\backslash L_{\hat{\res{z}}}^0),\R^d)}^2
&\leq
\int_{\GC}\int_{-\eps}^\eps\frac{|x_1|}{4\eps^2}
\Big(
\|\partial_{x_1}\hat\uu^+(\cdot,x')\|_{\rmL^2(I_\eps^+,\R^d)}
+\|\partial_{x_1}\hat\uu^-(\cdot,x')\|_{\rmL^2(I_\eps^-,\R^d)}
\Big)^2\,\mathrm{d}x_1\,\mathrm{d}x'
\\
&\leq \frac14\EEE\Big(
\|\partial_{x_1}\hat\uu^+\|_{\rmL^2(\OD^\eps,\R^d)}^2
+\|\partial_{x_1}\hat\uu^-\|_{\rmL^2(\OD^\eps,\R^d)}^2
\Big)\;\to0
\end{split}
\end{equation}
as $\eps\to0$. 
Putting together \eqref{Jump-3-1-1}, \eqref{Jump-3-1-2}, and \eqref{Jump-3-1-4} 
we conclude \eqref{Jump-3-1}. 
\par
{\bf Proof of \eqref{Jump-3-2}: }Here we follow the arguments of \cite[Lemma 5, Step 3, (3.43)-(3.44)]{MiRoTh10DDNE}. Therein, the term on the left-hand side of \eqref{Jump-3-2} was also shown to vanish as $\eps\to0$ due to the assumption $\gamma>1$. As we will see here below, the fact that now 
$\gamma=1$ prevents this term from vanishing and leads to the appearance of the interfacial jump term in the surface energy. 
\par
For $x'\in L_{\hat{\res{z}}}^0$  it holds \EEE $\hat\uu^+=(0,x')\not=\hat\uu^-(0,x')$ in general. Then we find
\begin{equation}
\label{Jump-3-2-1}
\begin{split}
\big|
\hat\uu^+(x_1,x')-\hat\uu^-(x_1,x')
\big|
&\leq 
\big|\JUMP{\hat\uu}(x')\big|
+\Big|\int_0^{x_1}\partial_\xi\hat\uu^+(\xi,x')\,\mathrm{d}\xi\Big|
+\Big|\int_0^{x_1}\partial_\xi\hat\uu^-(\xi,x')\,\mathrm{d}\xi\Big|
\\
&\leq\big|\JUMP{\hat\uu}(x')\big|+A(\partial_{x_1}\hat\uu^+,\partial_{x_1}\hat\uu^+)
\end{split}
\end{equation}
with $A(\partial_{x_1}\hat\uu^+,\partial_{x_1}\hat\uu^+)$ from \eqref{Jump-3-1-3}. 
Dividing by $2\eps$ and integrating over 
$\rmL^2((I_\eps^-\cup I_\eps^+)\times L_{\hat{\res{z}}}^0,\R^d)$ leads to 
\begin{equation}
\label{Jump-3-2-2}
\begin{split}
\|G_2^\eps\|_{\rmL^2((I_\eps^-\cup I_\eps^+)\times L_{\hat{\res{z}}}^0,\R^d)}^2
&\leq \frac{2}{4\eps}\|\JUMP{\hat\uu}\|_{\rmL^2(\GC,\R^d)}^2
+\frac{2}{4\eps^2}\int_{L_{\hat{\res{z}}}^0}\int_{-\eps}^\eps A(\partial_{x_1}\hat\uu^+,\partial_{x_1}\hat\uu^+)\,\mathrm{d}x_1\,\mathrm{d}x'
\\
&\leq
\frac{1}{2\eps}\|\JUMP{\hat\uu}\|_{\rmL^2(\GC,\R^d)}^2
+\frac12\EEE\|\partial_{x_1}\hat\uu^+\|_{\rmL^2(\OD^\eps,\R^d)}^2
+\frac12\EEE\|\partial_{x_1}\hat\uu^-\|_{\rmL^2(\OD^\eps,\R^d)}^2\,,
\end{split}
\end{equation}
where we used \eqref{Jump-3-1-4} to arrive at the second and the third term, and from there we also see that these two terms vanish as $\eps\to0$. Now \eqref{Jump-3-2} can be concluded thanks to the additional prefactor $\eps$ that stems from the fact that $\gr^\eps(\recop{\res{z}}{\eps}(\hat{\res{z}}))=\eps$ on $L_{\hat{\res{z}}}^0$. 
More precisely, together with  \eqref{Jump-3-2-2} and \eqref{Jump-3-1-2} we conclude 
\begin{equation*}
\begin{split}
\limsup_{\eps\to0}
\int_{L_{\hat{\res{z}}}^0}\int_{-\eps}^\eps
\eps\,\frac{\mu}{2}|\nabla\recop{\uu}{\eps}(\hat \uu)|^2
\,\mathrm{d}x_1\,\mathrm{d}x'
&\leq 
\limsup_{\eps\to0}
\eps\frac{\mu}{2} \big( 2 
\|G_1^\eps\|_{\rmL^2((I_\eps^-\cup I_\eps^+)\times L_{\hat{\res{z}}}^0,\R^d)}^2
+ 2  \|G_2^\eps\|_{\rmL^2((I_\eps^-\cup I_\eps^+)\times L_{\hat{\res{z}}}^0,\R^d)}^2
\big)
\\
&\leq
\limsup_{\eps\to0}
\eps \mu
\|G_1^\eps\|_{\rmL^2((I_\eps^-\cup I_\eps^+)\times L_{\hat{\res{z}}}^0,\R^d)}^2
+\limsup_{\eps\to0}
\eps \mu
\|G_2^\eps\|_{\rmL^2((I_\eps^-\cup I_\eps^+)\times L_{\hat{\res{z}}}^0,\R^d)}^2
\\
&\leq 
0+ \limsup_{\eps\to0} \eps\mu\Big(\frac{1}{2\eps}\|\JUMP{\hat\uu}\|_{\rmL^2(\GC,\R^d)}^2
+\frac12\|\partial_{x_1}\hat\uu^+\|_{\rmL^2(\OD^\eps,\R^d)}^2
+\frac12\|\partial_{x_1}\hat\uu^-\|_{\rmL^2(\OD^\eps,\R^d)}^2\Big)
\\
&=\int_{L_{\hat{\res{z}}}^0}
\frac{\mu}{2}\big|\JUMP{\hat\uu}\big|^2
\,\mathrm{d}x'
\,,
\end{split}
\end{equation*}
which gives \eqref{Jump-3-2} thanks to the brittle constraint \eqref{props-ens-new-2b-brittle}. 
\par
Hence the proof of statement \eqref{props-ens-new-2b} is finished.
\end{proof}

\section*{acknowledgements}    
The work of E.D. was supported by the Austrian Science Fund (FWF) through
projects 10.55776/Y1292, 10.55776/P35359, and 10.55776/F100800. 
For open access purposes, the authors have applied a CCBY public copyright license to any
accepted manuscript version arising from this submission. G.B.\ and R.R.\ acknowledge the support 
of GNAMPA (INDAM).  R.R.\ has also been supported by  an ERC grant (Project 101200514 — OPTiMiSE)\footnote{Funded by the European Union. Views and opinions expressed are, however, those of the author(s) only and do not necessarily reflect those of the European Union or the granting authority. Neither the European Union nor the granting authority can be held responsible for them}.
M.T.\ acknowledges the support by the German Research Foundation (DFG) within project 
B09 \emph{Materials with discontinuities on many scales} of CRC 1114 \emph{Scaling Cascades in Complex Systems} (project-number 235221301) 
and within project \emph{Nonlinear Fracture Dynamics: Modeling, Analysis, Approximation, and Applications} of Priority Programme SPP 2256 \emph{Variational Methods for Predicting Complex Phenomena in Engineering Structures and Materials} (project-number 441212523).   
           
           \medskip

    \bibliographystyle{alpha}
\bibliography{ricky_lit}

\end{document}